\documentclass[10pt]{amsart}
\usepackage[english]{babel}
\usepackage{amssymb}
\usepackage{amsmath}
\usepackage{srcltx}
\usepackage{lscape}

\newcommand{\can}{\overline{\phantom{x}}}

\newtheorem{dummy}{Dummy}

\newtheorem{lemma}[dummy]{Lemma}
\newtheorem{theorem}[dummy]{Theorem}
\newtheorem{proposition}[dummy]{Proposition}
\newtheorem{corollary}[dummy]{Corollary}

\theoremstyle{definition}
\newtheorem{definition}{Definition}

\newtheorem{example}[dummy]{Example}

\newtheorem{remark}[dummy]{Remark}

\newcommand{\ignore}[1]{}

\author{S. Pumpl\"un}
\email{susanne.pumpluen@nottingham.ac.uk}
\address{School of Mathematical Sciences\\
University of Nottingham\\
University Park\\
Nottingham NG7 2RD\\
United Kingdom
}



\begin{document}

\title[Semiassociative algebras]
{Nonassociative cyclic algebras and the semiassociative Brauer monoid}

\begin{abstract} We look at classes of semiassociative algebras, with an emphasis on those that canonically generalize associative (generalized) cyclic algebras, and at their behaviour in the semiassociative Brauer monoid defined by Blachar, Haile,  Matzri,  Rein, and  Vishne. A  possible way to generalize this monoid in characteristic $p$ that includes nonassociative differential algebras is  briefly considered.

\end{abstract}

\maketitle

%
\section*{Introduction}
%

Recently, semiassociative algebras and the semiassociative Brauer monoid denoted $Br^{sa}(F)$  were introduced by Blachar, Haile,  Matzri,  Rein, and  Vishne  \cite{BHMRV} as  canonical generalizations of associative central simple algebras and their Brauer group.
  Semiassociative algebras $A$ over  a field $F$ are $F$-central and are characterized by having an \'etale algebra $E$ contained in their nucleus, such that $A$ is cyclic and faithful
as an $E\otimes_F E$-module via the action $(e\otimes e')a=eae'$ for all $a\in A$, $e,e'\in E$. This definition makes it possible to use classical Brauer factor sets
\cite[Chapter 2]{J96} when developing the theory, and guarantees that the algebras are forms of skew matrix algebras, which are defined and investigated in depth \cite{BHMRV}.

Together with the tensor product, equivalence classes of semiassociative algebras over $F$ form a monoid that contains the classical Brauer group as its unique maximal subgroup.
 The skew matrix algebras now play the role of the classical matrices in the Brauer group. In particular, a semiassociative algebra is called split if it is isomorphic to a skew matrix algebra.
The authors state that ``the key example for semiassociative algebras are skew matrices''  \cite{BHMRV}.

In this paper we will look at another important example of semiassociative algebras; the nonassociative (generalized) cyclic algebras (and their opposite algebras).
 It is well known that (generalized) cyclic algebras play a prominent role in the structure theory of classical central simple algebras. Here, we look at the role  nonassociative (generalized) cyclic algebras play in the structure theory of semiassociative simple algebras.
  These algebras are canonical generalizations of associative cyclic algebras (respectively, of the generalized associative cyclic algebras  introduced by Jacobson \cite{J96}) over $F$.

 If $F$ has a cyclic Galois field extension $K/F$ of degree $n$, then there exist nonassociative cyclic algebras that have $K$ as their nucleus, and these are semiassociative algebras of degree $n$. More generally, for any field $F$ with cyclic Galois field extensions, there exist  nonassociative generalized cyclic algebras. These have a central simple algebra in their nucleus with center a separable field extension of $F$ and are semiassociative algebras as well. Both these types of semiassociative algebras are not semicentral, and thus in particular not homogeneous. They have infinite order in $Br^{sa}(F)$, even when they are not division algebras.  Nonassociative (generalized) cyclic algebras also nicely show that the splitting behaviour of an algebra in $Br^{sa}(F)$ (i.e., whether or not it will be split or split under a field extension) only depends on its nucleus.

We finish by suggesting possible generalizations of the semiassociative Brauer monoid, which allow us to include nonassociative (generalized) differential extensions as classes of algebras in the monoid, if the characteristic of $F$ is prime. Our motivation is that the associative differential algebras play an important role in the Brauer group, so it seems natural to try find a way to include them in some semiassociative Brauer monoid definition.

The structure of the paper is as follows: we collect the basic results needed in Section \ref{sec:prel}. In Section 2, we generalize the definition of nonassociative cyclic algebras $(K/F,\sigma, d)$ to include the case that $K/F$ is an \'etale extension, and the definition of
generalized nonassociative cyclic algebras $(D,\sigma, d)$ to include the case that the algebra $D$ employed in the construction with the skew polynomial $t^m-d\in D[t;\sigma]$ has zero divisors, collecting and generalizing several previously achieved results, including the explicit computation of the right nucleus for these algebras. We observe that the opposite algebra of a nonassociative cyclic algebra is semiassociative, but  need not be a nonassociative cyclic algebra again (Corollary \ref{cor:opp cyclic}).
  In Section \ref{sec:tensor}, we look at the  tensor product of a central simple algebra and a nonassociative cyclic algebra.  Cyclic algebras $(K/F,\sigma,d)$ of degree $n$ that are not associative are division algebras in many cases, e.g. for all prime $n$.
Some of the tensor products of nonassociative cyclic algebras that we consider yield examples of semiassociative algebras which again are division algebras, and are nonassociative generalized cyclic algebras.
 We  investigate the behaviour of nonassociative (generalized) cyclic algebras  in the Brauer monoid $Br^{sa}(F)$ in Section \ref{subsec:structure}, and briefly look at $Br^{sa}(\mathbb{R})$  and $Br^{sa}(\mathbb{F}_q)$.

When $F$ is a field of prime characteristic $p$, the definition of $Br^{sa}(F)$ may benefit from a generalization that includes a class of algebras that generalize algebras that are associative differential extensions \cite{J96}:
All associative central division algebras over a field $F$ of characteristic zero can
be constructed using differential polynomials  (Amitsur \cite{Am}, and later \cite{Hoe},
\cite[Sections 1.5, 1.8, 1.9]{J96}). The construction method is an analogue
to the  well-known crossed product construction, except that instead of  algebraic splitting fields
 it uses  splitting fields $K$, where $F$ is algebraically closed in $K$.
For  $p$-algebras over base fields of characteristic $p>0$ the construction employs
 differential polynomial rings $D[t;\delta]$ (where $D$ is a central division algebra over $C$), factoring out a two-sided ideal
generated by some suitable $f\in D[t;\delta]$.
This construction was generalized to the nonassociative setting in \cite{P16.0}. We briefly consider these algebras in characteristic $p$ and the pros and cons to include them in potential generalizations of $Br^{sa}(F)$ in the last two sections.

Nonassociative (generalized)  cyclic algebras already appeared in space-time block coding \cite{PU11,  PS15.4, P13, PS14}, and in $(f,\sigma,\delta)$-codes \cite{P.codes}.
In this paper we generalize their definition which previously usually employed skew polynomials in $D[t;\sigma]$ over division rings $D$, and drop the assumption that $D$ has no zero divisors.  Generalized Menichetti algebras \cite{M, P23} can be seen as generalizations of both crossed products and nonassociative cyclic algebras, and make up the second class of semiassociative algebras we present.

%
%

\section{Preliminaries} \label{sec:prel}

\subsection{Nonassociative algebras} \label{subsec:nonassalgs}


Let $F$ be a field. An $F$-vector space $A$ is an
\emph{algebra} over $F$, if there exists an $F$-bilinear map $A\times
A\to A$, $(x,y) \mapsto x \cdot y$, usually denoted  by juxtaposition, the  \emph{multiplication} of $A$. An algebra $A$ is called
\emph{unital} if there is an element in $A$, denoted by 1, such that
$1x=x1=x$ for all $x\in A$. We only consider unital algebras.

Associativity in $A$ is measured by the {\it left nucleus} ${\rm
Nuc}_l(A) = \{ x \in A \, \vert \, [x, A, A]  = 0 \}$, the {\it
middle nucleus} ${\rm Nuc}_m(A) = \{ x \in A \, \vert \,
[A, x, A]  = 0 \}$  and  the {\it right nucleus}
${\rm Nuc}_r(A) = \{ x \in A \, \vert \, [A,A, x]  = 0 \}$  of $A$, where $[x, y, z] =
(xy) z - x (yz)$ is the  {\it associator}. ${\rm
Nuc}_l(A)$, ${\rm Nuc}_m(A)$, and ${\rm Nuc}_r(A)$ are associative
subalgebras of $A$, and their intersection
 ${\rm Nuc}(A) = \{ x \in A \, \vert \, [x, A, A] = [A, x, A] = [A,A, x] = 0 \}$ is the {\it nucleus} of $A$.
${\rm Nuc}(A)$ is an associative subalgebra of $A$ containing $F1$
and $x(yz) = (xy) z$ whenever one of the elements $x, y, z$ is in
${\rm Nuc}(A)$. The {\it center} of $A$ is ${\rm C}(A)=\{x\in \text{Nuc}(A)\,|\, xy=yx \text{ for all }y\in A\}$ and $A$ is called \emph{($F$-)central} if
 it has center $F$.

 Multiplication on both sides make $A$ into a bimodule over its nucleus. Moreover, for every subalgebra $N$ of the nucleus the
$N$-bimodule structure of $A$ can be viewed as a left module structure over the ring $N^e=N\otimes_FN^{op}$.

An algebra $A\not=0$ is called a \emph{division algebra} if for any
$a\in A$, $a\not=0$, the left multiplication  with $a$, $L_a(x)=ax$,
and the right multiplication with $a$, $R_a(x)=xa$, are bijective.
If $A$ has finite dimension over $F$, $A$ is a division algebra if
and only if $A$ has no zero divisors \cite[pp. 15, 16]{Sch}.

An \emph{\'etale} algebra  over $F$ is a finite direct product of finite separable field extensions of $F$.

A {\it Galois $C_n$-algebra} over $F$ is an $n$-dimensional \'{e}tale algebra $K$ over $F$ endowed with an action by $C_n$,  where  $C_n$ is
 a group of $F$-automorphisms of $K$, such
that ${\rm Fix}(C_n)=\{x\in K\,|\, g(x)=x \text{ for all } g\in C_n \}=F$ \cite[(18.B.), p.~287]{KMRT}.
A Galois $C_n$-algebra structure on a field $K$ exists if and only if the field extension $K/F$ is Galois with Galois
group isomorphic to $C_n$  \cite[(18.16), p.~288]{KMRT}.

Let $K$ be a Galois $C_n$-algebra of dimension $n$, i.e.  a cyclic field extension of $F$ of degree $n$ with Galois group  ${\rm Gal}(K/F)=\langle\sigma \rangle$ or an \'etale algebra over $F$ with a cyclic automorphism group, with norm $N$ and trace $T$.  Let $\sigma$ generate its automorphism group.

\subsection{Semiassociative algebras}  (cf. \cite{BHMRV})\label{subsec:na}

A finite dimensional nonassociative $F$-central algebra $A$ is called \emph{semiassociative} if its nucleus has an \'etale $F$-subalgebra  $E$, such that $A$ is cyclic and faithful
over $E\otimes_F E$ via the action $(e\otimes e')a=eae'$ for all $a\in A$, $e,e'\in E$.
The dimension of a semiassociative algebra $A$ is a square \cite[Corollary 3.4 ]{BHMRV} and the root of the dimension of $A$ is called the
\emph{degree} of $A$. 

If $A$ is a nonassociative algebra containing an \'etale subalgebra $E$ in its nucleus, then any two
of the following conditions imply the third: $A$ is faithful over $E\otimes E^{\rm op}$, $A$ is
cyclic over $E\otimes E^{\rm op}$, and $dim A = (dim E)^2$ \cite[Remark 3.3]{BHMRV}.

 Every associative central simple algebra of degree $n$ has a maximal \'etale subalgebra $E$ of dimension $n$ and is semiassociative.
We call $A$ $E$-\emph{semiassociative} if $E$ is an \'etale $F$-subalgebra of its nucleus, such that $A$ is cyclic and faithful over $E^e=E\otimes_F E$.  The nucleus of a nonassociative algebra may contain more than one
\'etale subalgebra of the same dimension. However, if $A$ is a
semiassociative algebra with respect to one \'etale subalgebra of its nucleus, then it is
 semiassociative
with respect to all \'etale subalgebras of its nucleus \cite[Proposition 3.6]{BHMRV}.
A scalar extension of a semiassociative algebra is semiassociative \cite[Proposition 12.1]{BHMRV}.

A tensor $c_{ijk}$ of degree $n$ of $n\times n \times n$ scalars in $F$ is called a \emph{skew set $c$ of degree $n$}.
A skew set $c$ is called \emph{reduced} if $c_{iij}=c_{jii}=1$ for all $i,j$. Let $c$ be such a reduced skew set. Then the \emph{skew matrix algebra} $M_n(F;c)$ is the $F$-vector space with basis the matrix units $e_{ij}$ and multiplication given by
$e_{ij}e_{kl}=\delta_{jk}c_{ijl}e_{il}$.
Note that $M_n(F)=M_n(F;1)$.

 A semiassociative algebra $A$ is called \emph{split}, if it is a skew matrix
algebra.

 A field extension $K/F$ \emph{splits} a semiassociative algebra $A$, if
$A_K = A\otimes_F K $ is split. A semiassociative algebra of degree $n$ is split if and
only if $F^n$ is a unital subalgebra of the nucleus \cite[Proposition 7.2]{BHMRV}.
 Let $A$ be an $E$-semiassociative algebra with $E ={\rm Nuc}(A)$. Then a field extension $K/F$ splits $A$ if and only if it splits $E$ \cite[Corollary 7.5]{BHMRV}.
 If $K$ is a field that splits an \'etale subalgebra in the nucleus of an $n$-dimensional semiassociative algebra $A$ of degree $n$, and $F$ is an infinite field, then $K$ splits $A$ \cite[Theorem 7.1]{BHMRV}.

If A is semiassociative of degree $n$, then any $n$-dimensional
\'etale subalgebra $E$ of ${\rm Nuc}(A)$ is a maximal commutative subalgebra of $A$ \cite[Corollary 7.3]{BHMRV}.

Let $J({\rm Nuc}(A))$ denote the radical of the associative algebra ${\rm Nuc}(A)$. For a semiassociative
algebra $A$, the simple components of the semisimple quotient $\sigma(A)={\rm Nuc}(A)/J({\rm Nuc}(A))$
are called the \emph{atoms} of A.
A semiassociative algebra over $F$ is called \emph{semicentral}, if all of its atoms are $F$-central \cite[Definition 16.1]{BHMRV}.
 A semiassociative algebra is \emph{homogeneous} if it is
semicentral, and the atoms are all Brauer equivalent to each other \cite[Definition 17.2]{BHMRV}.

Two semiassociative algebras $A$ and $B$ over $F$ are called \emph{Brauer equivalent}, if there exist skew matrix algebras $M_n(F;c)$ and $M_m(F,c')$
such that $A\otimes_F M_n(F;c) \cong B \otimes_F M_m(F;c') $. The \emph{semiassociative Brauer monoid} $Br^{sa}(F)$ is the set of equivalence classes with respect to Brauer equivalence, with product $[A]^{sa}[B]^{sa}=[A\otimes_F B]^{sa}$ and unit element $[F]^{sa}$.
If $A$ is a homogeneous semiassociative algebra, and $D$ the (associative) underlying division algebra of its atoms, then there is a
decomposition $A\cong D\otimes_F M$, where $M$ is a skew matrix algebra and
$D$ is the unique member of minimal degree in the class $[D]^{sa}\in Br^{sa}(F)$ \cite[Proposition 18.2, Corollary 18.3]{BHMRV}.

\subsection{Nonassociative algebras obtained from skew polynomial rings}  (for details, cf. \cite{P.codes})  \label{sec:2}

Let $S$ be a unital associative noncommutative ring, $\sigma\in{\rm Aut}(S)$, and
$\delta:S\rightarrow S$ a \emph{$\sigma$-derivation}, i.e.
an additive map such that
$\delta(ab)=\sigma(a)\delta(b)+\delta(a)b$
for all $a,b\in S$.
The \emph{skew polynomial ring} $R=S[t;\sigma,\delta]$ is the set of skew polynomials
$a_0+a_1t+\dots +a_nt^n$
with $a_i\in S$, where addition is defined term-wise and multiplication by
$ta=\sigma(a)t+\delta(a) $ for all $ a\in S.$ We write $S[t;\sigma]=S[t;\sigma,0]$ and $S[t;\delta]=S[t;id,\delta]$.

 For $f(t)=a_0+a_1t+\dots +a_nt^n$ with $a_n\not=0$ define ${\rm deg}(f)=n$ and ${\rm deg}(0)=-\infty$.
Then ${\rm deg}(gh)\leq{\rm deg} (g)+{\rm deg}(h)$ for $f,g\in S[t]$ (with equality if $h$ or $g$ have an invertible leading coefficient, or if $S$ is a division ring).
 An element $f\in R$ is \emph{irreducible} in $R$ if it is not a unit and  it has no proper factors, i.e if there do not exist $g,h\in R$ with
 ${\rm deg}(g),{\rm deg} (h)<{\rm deg}(f)$ such
 that $f=gh$.
  We call $f\in R$ \emph{right-invariant}  polynomial, if $fR\subset Rf$.
  If $f$ is right invariant then $Rf$ is a two-sided ideal.

 Let $f\in R$ have degree $m$ and an invertible leading coefficient. Then for all $g(t)\in R$ of degree $l\geq m$,  there exist  uniquely determined $r,q\in R$ with ${\rm deg}(r)<{\rm deg}(f)$, such that $g(t)=q(t)f(t)+r(t).$
This generalizes the right division algorithm in $R$ that is well-known when $S$ is a division ring \cite[p.~6]{J96}.

 From now on we assume that $f\in R=S[t;\sigma,\delta]$ is monic of degree $m$.
Let ${\rm mod}_r f$ denote the remainder of right division by $f$.
Since the remainder is uniquely determined, the skew polynomials of degree less that $m$ canonically represent the
elements of the left  $S[t;\sigma,\delta]$-module $S[t;\sigma,\delta]/ S[t;\sigma,\delta]f$.

 The additive group $\{g\in R\,|\, {\rm deg}(g)<m\}$,  together with the multiplication
$g\circ h=gh \,\,{\rm mod}_r f$
 for all $g,h\in R$ of degree less than $m$, is a unital nonassociative algebra over $S_0=\{a\in S\,|\, ah=ha \text{ for all } h\in S_f\}$,
denoted by  $S_f$ or $R/Rf$ (here, $gh \,\,{\rm mod}_r f $ is the remainder of $gh$ after right multiplication by $f$). The construction for $S$ a division algebra goes back to  \cite{P66}, and theses algebras are called \emph{Petit algebras}.
$S_0$ is a commutative subring of $S$, and if $S$ is a division algebra, it is a subfield of $S$.
$S_f$ is associative if and only if $Rf$ is a two-sided ideal.
If $S_f$ is not associative then
$S\subset{\rm Nuc}_l(S_f),$ $S\subset{\rm Nuc}_m(S_f)$
(if $S$ is a division ring, the inclusions become equalities),
and the eigenspace of $f$ is the right nucleus:
${\rm Nuc}_r(S_f)=\{g\in R\,|\, {\rm deg}(g)<m \text{ and }fg\in Rf\}.$

 If $f\in S[t;\sigma,\delta]$ is reducible then $S_f$ contains zero divisors.
If $S$ is a division ring, then
$S_f$ has no zero divisors if and only if $f$ is irreducible.

If $Rf$ is a two-sided ideal in $R$ (i.e., $f$ is \emph{(right)-invariant}) then $S_f$ is
the   associative quotient algebra
 obtained by factoring out the ideal generated by a two-sided $f\in S[t;\sigma,\delta]$.

For all $g\in R$ of degree $s\geq m$,  there also exist  uniquely determined $r,q\in R$
with ${\rm deg}(r)<{\rm deg}(f)$, such that
$g(t)=f(t)q(t)+r(t).$ Let ${\rm mod}_l f$ denote the remainder of left division by $f$. Then the additive group $\{g\in R\,|\, {\rm deg}(g)<m\}$  together with the multiplication
$g\diamond h=gh \,\,{\rm mod}_l f $
defined for all $g,h\in R$ of degree less than $m$, is also a unital nonassociative algebra $\,_fS$ over $S_0$  denoted by $R/fR$ (here, $gh \,\,{\rm mod}_l f $ is the remainder of $gh$ after left multiplication by $f$). Moreover,
the canonical anti-automorphism
 $$\psi: S[t;\sigma,\delta]\rightarrow S^{op}[t;\sigma^{-1},-\delta\circ\sigma^{-1}],\quad \psi(\sum_{k=0}^{n}a_kt^k)=\sum_{k=0}^{n}(\sum_{i=0}^{k}\Delta_{n,i}(a_k))t^k$$
 induces an anti-automorphism
 between the algebras $S_f=S[t;\sigma, \delta]/ S[t;\sigma,\delta]f$
  and
$\,_{\psi(f)}S=S^{op}[t;\sigma^{-1},-\delta\circ\sigma^{-1}]/\psi(f) S^{op}[t;\sigma^{-1},-\delta\circ\sigma^{-1}],$
so that $S_f^{op}=\,_{\psi(f)}S$ (\cite[Corollary 4]{LS} holds for any base field). Here, $\Delta_{n,j}$ is defined recursively via
$\Delta_{n,j}=\delta(\Delta_{n-1,j})+\sigma (\Delta_{n-1,j-1}),$
with $ \Delta_{0,0}=id_S$, $\Delta_{1,0}=\delta$, $\Delta_{1,1}=\sigma $ and so $\Delta_{n,j}$ is the sum of all polynomials in $\sigma$ and $\delta$
of degree $j$ in $\sigma$ and degree $n-j$ in $\delta$ \cite[p.~2]{J96}.
If $\delta=0$, then $\Delta_{n,j}=\sigma^n$.

 %
 %

 \section{Nonassociative (generalized) cyclic algebras and generalized Menichetti algebras} \label{sec:nonasscyclic1}

 \subsection{Nonassociative cyclic algebras}

The equivalence class of a homogeneous semiassociative algebra in $Br^{sa}(F)$ is represented by a unique central associative division algebra \cite{BHMRV}. However, we will see now that
the question whether an algebra is a  division algebra or not is much less important in $Br^{sa}(F)$.

  Nonassociative cyclic algebras of degree $n$ are  canonical generalizations of associative cyclic algebras of degree $n$ and
 were first introduced over finite fields by Sandler \cite{S}.
  Indeed, nonassociative quaternion algebras (where $n=2$) were the first known example of a nonassociative division algebra \cite{D}.
 Over arbitrary fields they were investigated by Steele \cite{S12, S13} (Steele studied the opposite algebras of the nonassociative algebras we  define here, but used our notation).

\begin{definition}
Let $K$ be a Galois $C_n$-algebra over $F$ with ${\rm Aut}_F(K)=\langle \sigma\rangle$
 (i.e., $K$ has dimension $n$ and $\sigma\in {\rm Aut}_F(K)$ has order $n$).
 Let $f(t)=t^n-d\in K[t;\sigma]$ with $d\in K^\times$. The $F$-central algebra $S_f=K[t;\sigma]/K[t;\sigma]f$ is called a
\emph{nonassociative cyclic algebra} over $F$ and denoted by $(K/F,\sigma, d)$.
\end{definition}

This definition generalizes the one used in most papers, where $K$ is a cyclic Galois field extension of degree $n$ (for an earlier generalization, cf. \cite{P.codes}).  If $(K/F,\sigma, d)$ is not associative then the algebra $(K/F,\sigma, d)$ is sometimes also called a \emph{proper} nonassociative cyclic algebra, if it not clear from the context if we look at a classical associative cyclic algebra $(K/F,\sigma, d)$ or not.

If  $K/F$ is a cyclic Galois field extension of degree $n$ with Galois group ${\rm Gal}(K/F)=\langle \sigma\rangle$, then $(K/F,\sigma, d)$ is a classical associative cyclic algebra over $F$ of degree $n$ if $d\in F^\times$, and a nonassociative cyclic algebra as defined in  \cite{BP, S, S12}, if $d\in K\setminus F$.  We note that with our more general definition, we now have for instance that $(K/F,\sigma, d)\otimes_F K\cong (K\otimes_F K, \sigma, d)$ with $\sigma$ denoting the canonical extension $\sigma\otimes id$ of $\sigma$ to $K\otimes_F K$. For more details and proofs in the case that $K$ is a division algebra, cf. \cite{BP, S12}.

The easiest example of a proper nonassociative cyclic division algebra is a nonassociative quaternion division algebra $(K/F,\sigma, d)$, where $K/F$ is a quadratic field extension, and $d\in K\setminus F$. This is, up to isomorphism, also the only simple $K$-semiassociative division algebra of degree two that is not associative \cite{W}. This algebra and the simple skew matrix algebra that represents it when it splits over the field extension $K$  are presented in \cite{W}. 

\begin{theorem}\label{thm:nuc}
 Let $K$ be a Galois $C_n$-algebra over $F$ with ${\rm Aut}_F(K)=\langle \sigma\rangle$
  and $(K/F,\sigma, d)$ be a proper nonassociative cyclic algebra, i.e. $d\in K\setminus F$. Let $H=\{\tau\in G\,|\,\tau(d)=d \}$. Then $H=\langle \sigma^{s}\rangle$ for some integer $s$ such that $n=sr$. Put $E={\rm Fix}(\sigma^s)$. Then the following holds:
  \\ (i) ${\rm Nuc}_r((K/F,\sigma, d))=K[t;\sigma^s]/K[t;\sigma^s](t^n-d) $ is a Petit algebra with center $E$. In particular, $K\subset {\rm Nuc}_r((K/F,\sigma, d))$ and hence $K\subset{\rm Nuc}((K/F,\sigma, d))$. Moreover,  if $n$ is prime then ${\rm Nuc}_r((K/F,\sigma, d))=K$.
  \\ (ii) If $K/F$ is a Galois field extension, then
 $${\rm Nuc}_r((K/F,\sigma, d))=(K/E,\sigma^s, d)$$
 is a cyclic algebra of degree $r$ over the field $E={\rm Fix}(\sigma^s)$, where $[E:F]=s$. This implies that  ${\rm Nuc}_r((K/F,\sigma, d))={\rm Nuc}((K/F,\sigma, d))=K$.
\end{theorem}

\begin{proof}
Let $K/F$ be a cyclic Galois field extension of degree $n$, then ${\rm Nuc}_r((K/F,\sigma, d))=K\oplus Kt^s \oplus \dots \oplus Kt^{(r-1)s}$ \cite[Proposition 3.2.3]{S13}.
By \cite[Theorem 5.1]{S12}, the linear subspace $K\oplus Kt^s \oplus \dots \oplus Kt^{(r-1)s}$ is the cyclic subalgebra $(K/E,\sigma^s, d)$ of degree $r$ over $E={\rm Fix}(\sigma^s)$, where $d\in E$ and $[E:F]=s$. A close look at the proofs of both results shows that neither of them uses that $K/F$ is a field extension; both analogously hold when $K$ is  a Galois $C_n$-algebra over $F$ with ${\rm Aut}_F(K)=\langle \sigma\rangle$.
In this later case, $E$ need not be a field extension, and the linear subspace $K\oplus Kt^s \oplus \dots \oplus Kt^{(r-1)s}$ is the Petit algebra $K[t;\sigma^s]/K[t;\sigma^s](t^n-d) $.

We know that $K\subset {\rm Nuc}_l((K/F,\sigma, d))$ and $K\subset {\rm Nuc}_m((K/F,\sigma, d))$ (with equalities when $K/F$ is a Galois field extension) \cite{P.codes}. Since $K\subset {\rm Nuc}_r((K/F,\sigma, d))$, we showed that $K\subset{\rm Nuc}((K/F,\sigma, d))$, with equality when $K/F$ is a Galois field extension. This proves the assertions (i) and (ii).
\end{proof}

\begin{corollary}\label{cor:opp cyclic}
Let $K$ be a Galois $C_n$-algebra over $F$ with ${\rm Aut}_F(K)=\langle \sigma\rangle$
and $d\in K\setminus F$. 
\\ (i) The algebra $(K/F,\sigma, d)^{op}$ is isomorphic to the Petit algebra $\,_{\psi(f)}S$ we obtain by using left instead of right division by $t^m-\sigma^{-m}(d)\in K[t;\sigma^{-1}]$.
 \\ (ii) Let $K/F$ be a cyclic Galois field extension of degree $n$. Let $H=\{\tau\in G\,|\,\tau(d)=d \}$, that means $H=\langle \sigma^s\rangle$ for some integer $s$ with $1<s<n$. Then $(K/F,\sigma, d)^{op}$ is not isomorphic to a nonassociative cyclic algebra over $F$.
\end{corollary}
    
\begin{proof}
(i) By \cite[Corollary 4]{LS}, the map
$\psi: K[t;\sigma]\rightarrow K^{op}[t;\sigma^{-1}],$ 
$$\psi(\sum_{k=0}^{n}a_kt^k)=\sum_{k=0}^{n}(\sum_{i=0}^{k}\Delta_{n,i}(a_k))t^k,$$
 induces an anti-automorphism between the nonassociative cyclic algebra $S_f
 =(K/F,\sigma, d)$  and the algebra
$\,_{\psi(f)}S=K^{op}[t;\sigma^{-1}]/\psi(f) K^{op}[t;\sigma^{-1}],$
which means that $(K/F,\sigma, d)^{op}=\,_{\psi(f)}S$. Here, $\psi(f)(t)=t^m-\sigma^{-m}(d)\in K^{op}[t;\sigma^{-1}]=K[t;\sigma^{-1}]$.
\\ (ii) Since $1<s<n$, we know that  ${\rm Nuc}_r((K/F,\sigma, d))=(K/E,\sigma^s, d)$ is a cyclic algebra of degree $r>1$ over the field $E$, where $E$ is a proper intermediate field of $K/F$ by Theorem \ref{thm:nuc}. Since the right nucleus of $(K/F,\sigma, d)$ is the left nucleus of $(K/F,\sigma, d)^{op}$,
thus the left nucleus ${\rm Nuc}_l((K/F,\sigma, d)^{op})$ is an associative cyclic 
algebra of degree $r>1$ over $E$ and therefore it is is unequal to the left nucleus of any nonassociative cyclic algebra
$(K'/F,\sigma', d')$:  ${\rm Nuc}_l((K'/F,\sigma', d'))=K'$. (Note that the middle nucleus of a proper nonassociative cyclic algebra is $K$, so if $A$ is a nonassociative cyclic algebra isomorphic to $(K/F,\sigma, d)^{op}$, $A$ must involve the same  extension $K/F$, up to isomorphism, so here in fact we can even assume that $K=K'$.)

\end{proof}

Nonassociative cyclic algebras (and their opposite algebras) are important examples of semiassociative (division)  algebras that are not semicentral, thus in particular not homogeneous:

\begin{proposition}
(i) Every nonassociative cyclic algebra $(K/F,\sigma, d)$ over $F$ is $K$-semiassociative of degree $n$ and thus semiassociative.
\\ (ii) Let $d\in K\setminus F$. Then $(K/F,\sigma, d)$ is split if and only if $K=F^n$.
\\ (iii)  $(K/F,\sigma, d)\otimes_F K$ splits.
\\ (iv) \cite{S12} Let  $K/F$ be a cyclic Galois field extension of degree $n$. Then $(K/F,\sigma, d)$ is a division algebra for all $d\in K\setminus F$, such that 1, $d,\dots, d^{n-1}$ are linearly independent over $F$. If $K/F$ has prime degree then $(K/F,\sigma, d)$ is a division algebra for all $d\in K\setminus F$.
\\ (v) Let $K/F$ be a cyclic Galois field extension, then for all $d\in K\setminus F$, $(K/F,\sigma, d)$ is not semicentral, ${\rm Nuc}(A)/J({\rm Nuc}(A))=K$ and  $(K/F,\sigma, d)$ is not homogeneous.
\end{proposition}

\begin{proof}
The proof of (i), (ii), (iii) is straightforward employing results from \cite{BHMRV} listed in Section \ref{subsec:na}. For $n=2$, (i) was already pointed out in \cite{BHMRV}. 
\\ (v) Since $K= {\rm Nuc}((K/F,\sigma, d))$ we know that  $A=(K/F,\sigma, d)$ is not semicentral  for all $d\in K\setminus F$.
Since $(K/F,\sigma,d)$  is not semicentral, it does not lie in the similarity class of any $F$-central simple algebra $B$ in $Br^{sa}(F)$ and is thus not homogeneous \cite{BHMRV}. The semisimple quotient ${\rm Nuc}(A)/J({\rm Nuc}(A))=K$ has $K$ as its only simple component (i.e. atom), and $K$ is not $F$-central.
\end{proof}

\begin{example}
The split nonassociative quaternion algebra defined in \cite{W} is simple, has nucleus $F\times F$, the basis $e_1,e_2,e_3,e_4$, with $e_1=(1,0)$, $e_2=(0,1)$ and unit element $e_1+e_2$, and its multiplication is given by $e_1e_1=e_1$, $e_1e_2=0=e_2e_1$, $e_1e_3=e_3$, $e_1e_4=0$, $e_2e_2=e_2$,
$e_2e_3=0$, $e_2e_4=e_4$, $e_3e_1=0$, $e_3e_2=e_3$, $e_3e_4=e_1$, $e_4e_1=e_4$, $e_4e_2=0$, $e_4e_3=\lambda e_2$, $e_4e_4=0$, with $\lambda\not=0,1$.
This is the skew matric algebra $M_2(F;c)$ of degree 2 with the reduced tensor given by $c_{212}=\lambda$, $c_{121}=1$, where $e_{11}=e_1, e_{22}=e_2, e_3=e_{12}, e_4=e_{21}$. As already noted in  \cite{W}, $M_2(F;c)=M_2(F)$ when $\lambda=1$, and for a nonassociative quaternion algebra $(K/F,\sigma,d)$ we have $(K/F,\sigma,d)\otimes_FK \cong M_2(F;c)$ with $\lambda=\sigma(d)/d$ in the reduced tensor $c$.
This is a skew matrix algebra of the type mentioned in \cite[Example 6.12 (3)]{BHMRV}.
\end{example}

\begin{remark}
Let $K$ be an \'etale algebra of dimension $n$ over $F$ and $\sigma\in {\rm Aut}_F(K)$ of order $n$.
 Let $f(t)=t^n\in K[t;\sigma]$. Then $S_f=K[t;\sigma]/K[t;\sigma]f$ is an associative algebra over $F$ which is semiassociative but not simple; its semisimple quotient is $K$. Abusing notation we denote it by $(K/F,\sigma, 0)$ (cf. \cite[Remark 3.8]{BHMRV} for $n=2$). It can be viewed as a generalization of the dual quaternion algebra $\mathbb{H}\otimes_\mathbb{R}\mathbb{D}$ used in physics, where $\mathbb{D}=\mathbb{R}[t]/(t^2)$ are the dual numbers.
\end{remark}

 \subsection{Nonassociative generalized cyclic algebras} \label{sec:nonasscyclic}   (For  details on the case that $B$ is a division algebra, cf. \cite{BP})

 Let $B$ be a central simple algebra over $C$ (i.e., $C$-central) of degree $n$, and $\sigma\in {\rm Aut}(B)$ such that
$\sigma|_{C}$ has finite order $m$  and put $F={\rm Fix}(\sigma)\cap C$. Then $C/F$ is a cyclic Galois field extension of
 degree $m$ with $\mathrm{Gal}(C/F) = \langle \sigma|_{C} \rangle$.

\begin{definition}
Let $f(t)=t^m-d\in B[t;\sigma]$, $d\in B^\times$. We call
$(B,\sigma, d)=B[t;\sigma]/B[t;\sigma]f$
 a \emph{nonassociative generalized  cyclic algebra (of degree $mn$)} over $F$. If $(B,\sigma, d)$ is not associative, we also call $(B,\sigma, d)$ a \emph{proper} nonassociative generalized  cyclic algebra.
 \end{definition}

This definition generalizes the definition of both a nonassociative and an associative generalized  cyclic algebra in \cite{BP} (see \cite[p.~19]{J96} for the associative case), which also assumed that $B$ is a division algebra.

  The algebra $(B,\sigma, d)$ has dimension $m^2n^2$ over $F$ and is $F$-central.
If $d\in F^\times$ and $B$ is a division algebra, then $(B,\sigma, d)$ is a classical associative generalized cyclic algebra over $ F$ of degree $mn$. Indeed,  $(B,\sigma, d)$ is associative if and only if $d\in F$.

For a proper generalized cyclic algebra $(B,\sigma, d)$, we have $B\subset{\rm Nuc}_l((B,\sigma, d))={\rm Nuc}_m((B,\sigma, d))$  with equality when $B$ is a division algebra. Moreover,
 if $d\in C\setminus F$ then also $B\subset {\rm Nuc}_r((B,\sigma, d))$, i.e. $B\subset{\rm Nuc}((B,\sigma, d))$ with equalities when $B$ is a division algebra  \cite{BP18}.

In particular, if $B=C$, $C/F$ is a cyclic Galois extension of degree $m$ with Galois group generated by
$\sigma$ and $ f(t)=t^m-d\in C[t;\sigma]$,  we obtain the nonassociative cyclic algebra $(C/F,\sigma,d)$ as a special case.

If $D$ is a division algebra of degree $n$ over $C$, then
$(D, \sigma, d)$ is a division algebra over
$F_0$ if and only if $t^m-d\in D[t;\sigma]$ is irreducible \cite[(7)]{P66}.
 We know that
 $t^2-d\in D[t;\sigma]$ is irreducible  if and only if $\sigma(z)z\not=d$ for all $z\in D$, and
 $t^3-d\in D[t;\sigma]$ is irreducible  if and only if $d\not=\sigma^2(z)\sigma(z)z$
 for all $z\in D$ 
(cf. \cite{P66, P16}, and \cite[Theorem 3.19]{CB}, \cite{BP18}). More generally, if $F$
contains a primitive $m$th root of unity and $m$ is prime then $t^m-d\in D[t;\sigma]$ is irreducible if and only if
$d\not=\sigma^{m-1}(z)\cdots\sigma(z)z$ for all $z\in D$ (\cite[Theorem 3.11]{CB}, see also \cite[Theorem 6]{P16}). This generalizes the equivalent condition in the associative setup.

\begin{lemma}
Let $K$ be a maximal \'etale subalgebra of $B$ of dimension $n$.
\\ (i) 
For all  $d\in C\setminus F$, the algebra $(B,\sigma, d)$ contains the \'etale algebra $K/F$ of dimension $mn$ in its nucleus.
\\ (ii)
For all $d\in C^\times$, $(B,\sigma, d)$ is a $K$-semiassociative algebra over $F$ of degree $mn$. If $d\in C\setminus F$ then
$(B,\sigma, d)$ is not  semicentral and not homogeneous.
\\ (iii) Suppose that $m$ is prime and that $F$
contains a primitive $m$th root of unity.  Let $d\in C\setminus F$. Assume that $B=D$ is a division algebra and that $d\not=\sigma^{m-1}(z)\cdots\sigma(z)z$ for all $z\in D$. Then $D$ is the only atom of $(D,\sigma, d)$.

\end{lemma}

\begin{proof}
(i) Since $K\subset B\subset {\rm Nuc}((B,\sigma, d))$ this is obvious.
\\ 
(ii) If $d\in F^\times$ then $(B,\sigma, d)$ is an associative central simple algebra over $F$ and trivially semiassociative.
For all $d\in C\setminus  F$  the \'etale algebra $K/F$ of degree $mn$ lies in ${\rm Nuc}((B,\sigma, d))$. The rest is a straightforward calculation as well:  $(B,\sigma, d)$ is a faithful $K^e$-module, thus cyclic
  as a $K^e$-module.
In particular, since $B\subset {\rm Nuc}((B,\sigma, d))$ is an $C$-central simple algebra, $(B,\sigma,d)$ is not semicentral and therefore also not homogeneous.
\\ (iii)  By our assumptions, $f(t)=t^m-d\in D[t;\sigma]$ is irreducible (\cite[Theorem 3.11]{CB}) and therefore $A=(D,\sigma, d)$ is a division algebra.  The semisimple quotient ${\rm Nuc}(A)/J({\rm Nuc}(A))=D$ has $D$ as its only simple component (i.e. as its only atom), and $D$ is $C$-central and not $F$-central.

\end{proof}

\begin{theorem}\label{thm:7}
Let $B=D$ be a $C$-central division algebra  of degree $n$ and  $d\in C\setminus F$. Then
  $C/F$ is a cyclic Galois field extension of degree $m$ with Galois group $G=\langle \sigma|_C\rangle$. Write $\sigma=\sigma|_C$ for ease of notation. Let $H=\{\tau\in G\,|\,\tau(d)=d \}$. Then $H=\langle \sigma^s\rangle$ for some integer $s$ such that $m=sr$ and
 $${\rm Nuc}_r((D,\sigma, d))=D[t;\sigma^s]/(t^m-d).$$
 In particular, if $m$ is prime then ${\rm Nuc}_r((D,\sigma, d))=D$.
\end{theorem}

\begin{proof}
By an analogous proof as the one for \cite[Proposition 3.2.3]{S13}, it is easy to see that ${\rm Nuc}_r((D,\sigma, d))=D\oplus Dt^s \oplus \dots \oplus Dt^{(r-1)s}$ as a left $D$-module.
By an analogous proof as the one for  \cite[Theorem 5.1]{S12}, the linear subspace $D\oplus Dt^s \oplus \dots \oplus Dt^{(r-1)s}$, together with the inherited algebra multiplication, is the associative subalgebra $D[t;\sigma^s]/(t^m-d)$
of dimension $n^2[C:F]r=n^2mr$ over $F$. Here, $d\in E={\rm Fix}(\sigma^s)$, $E$ is a proper subfield of $C/F$ and $D[t;\sigma^s]/(t^m-d)$ has center $E$.
\end{proof}

\begin{corollary}\label{cor:8}
Let $B=D$ be a division algebra  and suppose that $d\in C\setminus F$, $G=\langle \sigma|_C\rangle$, and $H=\{\tau\in G\,|\,\tau(d)=d \}$,  so that there exists an integer $s$ with $1<s<n$ with $H=\langle \sigma^s\rangle$. Then $(D,\sigma, d)^{op}$ is not isomorphic to a nonassociative generalized cyclic algebra.
\end{corollary}

\begin{proof}
By Theorem \ref{thm:7}, ${\rm Nuc}_r(((D,\sigma, d))=D[t;\sigma^s]/(t^m-d)$ is a noncommutative algebra  over the field $E$ unequal to $D$, where $E$ is a proper intermediate field of $C/F$. Since the right nucleus of $(D,\sigma, d)$ is the left nucleus of $(D,\sigma, d)^{op}$,
thus the left nucleus ${\rm Nuc}_l((D,\sigma, d)^{op})$ is unequal to the left nucleus of any nonassociative cyclic algebra
$(D',\sigma', d')$:  ${\rm Nuc}_l((D',\sigma', d'))=D'$. (Note that the middle nucleus of a proper generalized nonassociative cyclic algebra $(D,\sigma, d)$ is $D$, so if $A$ is a nonassociative cyclic algebra isomorphic to $(D,\sigma, d)^{op}$, $A$ must involve the same central simple algebra $D$, up to isomorphism, so here in fact we can even assume that $D=D'$.)
\end{proof}

Note that $(D,\sigma, d)^{op}$ is also semiassociative, however, and again not semicentral and not homogeneous.

\subsection{Menichetti algebras} \cite{M, P23}

Let  $K/F$ be a Galois field extension of $F$ of degree $m$ with ${\rm Gal}(K/F)=\{ \tau_0,\dots,\tau_{m-1} \}$.
Let $k_i\in K^\times$, $i\in 0,\dots,m-1$, and let
$$c_{i,j}=k_0^{-1}k_1^{-1} \cdots    k_{j-1}^{-1}   k_ik_{i+1}\cdots k_{i+j-1}$$
for all $i,j\in \mathbb{Z}_m$.
Let $z_0,\dots,z_{m-1}$ be an $F$-basis of $K^m$ and define a multiplication on $K^m$ via
$$(az_i)\cdot  (bz_j)=\tau_j(a)b (z_i\cdot  z_j), \quad
z_i\cdot  z_0=z_0\cdot  z_i =z_i \text{ for all } i\in\mathbb{Z}_m,$$
$$z_i\cdot  z_j=c_{ji}z_{i+j} \text{ for all } i\in\mathbb{Z}_m\setminus \{0\}$$
for all $a,b\in K$.
Then $K^m$ is a nonassociative unital algebra over $F$ of dimension $m^2$ we denote $(K/F, k_0,\dots,k_{m-1})$, and call a \emph{Menichetti algebra}, as the construction generalizes \cite{M}.
Define
\[ M(x_0,\dots,x_{m-1})=\left [\begin {array}{cccccc}
x_0 &  c_{m-1,1}\tau_1(x_{m-1}) & ... & c_{1,m-1} \tau_{m-1}(x_{1})\\
x_1 & \tau_1(x_0) & ... &  c_{2,m-1} \tau_{m-1}(x_{2})\\
x_2 & c_{1,1}\tau_1(x_{1}) & \tau_2(x_0) & c_{3,m-1} \tau_{m-1}(x_{3})\\
...& ...  & ... & ...\\
x_{m-2} &  c_{m-3,1}\tau_1(x_{m-3})  & ...& c_{m-1,m-1} \tau_{m-1}(x_{m-1})\\
x_{m-1} &  c_{m-2,1}\tau_1(x_{m-2})  & ...& \tau_{m-1}(x_{0})\\
\end {array}\right ]
 \]
and identify $x_0z_0+\dots +x_{m-1}z_{m-1}$ with $(x_0,\dots,x_{m-1})$, $x_i\in K$, then
 $$(x_0,\dots,x_{m-1})\cdot  (y_0,\dots,y_{m-1})=M(x_0,\dots,x_{m-1})(y_0,\dots,y_{m-1})^t.$$
 It is easy to see that $K\subset {\rm Nuc}((K/F, k_0,\dots,k_{m-1}))$ and that $(K/F, k_0,\dots,k_{m-1})$ is a semiassociative algebra over $F$.
 The algebras $(K/F, k_0,\dots,k_{m-1})^{op}$ generalize  nonassociative cyclic algebras $(K/F,\sigma,d)$ \cite{P23}.

\subsection{Generalized Menichetti algebras} \cite{P23}

 Let $D$ be a central simple algebra over $C$  of degree $n$. Let $\sigma\in {\rm Aut}(D)$ such that
$\sigma|_{C}$ has finite order $m$, and put $F={\rm Fix}(\sigma)\cap C$. Assume that
  $C/F$ is a cyclic Galois field extension of degree $m$ with $\mathrm{Gal}(C/F) = \langle \sigma|_{C} \rangle$
 (this is automatically satisfied, if $D$ is a division algebra).
Let $k_i\in C^\times$, $i\in 0,\dots,m-1$, and
$$c_{i,j}=k_0^{-1}k_1^{-1} \cdots    k_{j-1}^{-1}   k_ik_{i+1}\cdots k_{i+j-1}$$
for all $i,j\in \mathbb{Z}_m$.
Let  $z_0,\dots,z^{m-1}$ be a $C$-basis of $D^m$ and define a multiplication on $D^m$ via
$$(az_i)\cdot  (bz_j)=\sigma^j(a)b (z_i\cdot  z_j),$$
$$z_i\cdot  z_0=z_0\cdot  z_i =z_i \text{ for all } i\in\mathbb{Z}_m,\quad z_i\cdot  z_j=c_{ji}z_{i+j} \text{ for all } i\in\mathbb{Z}_m\setminus \{0\}$$
for all $a,b\in D$. This yields a nonassociative unital algebra over $F$ of dimension $n^2m^2$ that we denote by $(D, \sigma, k_0,\dots,k_{m-1})$ and call a \emph{generalized Menichetti algebra of degree $mn$}. Its multiplication can be written as
$$(x_0,\dots,x_{m-1})\cdot  (y_0,\dots,y_{m-1})=M(x_0,\dots,x_{m-1})(y_0,\dots,y_{m-1})^t$$
for all $x_i,y_i\in D$, with
\[ M(x_0,\dots,x_{m-1})=\left [\begin {array}{cccccc}
x_0 &  c_{m-1,1}\sigma(x_{m-1}) & ... & c_{1,m-1} \sigma^{m-1}(x_{1})\\
x_1 & \sigma(x_0) & ... &  c_{2,m-1} \sigma^{m-1}(x_{2})\\
x_2 & c_{1,1}\sigma(x_{1}) & \sigma^2(x_0) & c_{3,m-1} \sigma^{m-1}(x_{3})\\
...& ...  & ... & ...\\
x_{m-2} &  c_{m-3,1}\sigma(x_{m-3})  & ...& c_{m-1,m-1} \sigma^{m-1}(x_{m-1})\\
x_{m-1} &  c_{m-2,1}\sigma(x_{m-2})  & ...& \sigma^{m-1}(x_{0})\\
\end {array}\right ].\]
 Generalized Menichetti algebras $(D, \sigma, k_0,\dots,k_{m-1})^{op}$ can be seen as generalizations of the generalized cyclic algebras $(D^{op},\sigma,d)$ for $d\in C$.

  Let $E$ be a maximal \'etale subalgebra of $D$ of dimension $n$. Then
  $(D, \sigma, k_0,\dots,k_{m-1})$ has the \'etale algebra $E/F$ of dimension $mn$ in its nucleus.  $(D, \sigma, k_0,\dots,k_{m-1})$ is a faithful $E^e$-module,  is cyclic
  as a $E^e$-module \cite[Remark 3.3]{BHMRV} and thus a semiassociative algebra \cite{BHMRV}.

Let now $K/C$ be a cyclic field extension of degree $m$  with ${\rm Gal}(K/C)=\langle \sigma\rangle$.
Let $D_0$ be a central simple algebra over $C$ of degree $n$, and put $D=D_0\otimes_C K$. Let $\widetilde{\sigma}$ be the unique extension of $\sigma$ to $D$ such that $\widetilde{\sigma}|_{D_0}=id_{D_0}$.   Then it is straightforward to check that
$$A=D_0\otimes_F (K/F, k_0,\dots,k_{m-1})\cong (D_0\otimes_C K, \widetilde{\sigma}, k_0,\dots,k_{m-1})$$
is a nonassociative generalized Menichetti algebra over $F$ of degree $mn$.
Since ${\rm Nuc}(A)=D_0\otimes {\rm Nuc}((K/F, k_0,\dots,k_{m-1}))$, we have
$D\subset {\rm Nuc}(A)$. $A$ is a semiassociative algebra.


 \section{The tensor product of an associative algebra and a nonassociative cyclic algebra}\label{sec:tensor}


 The tensor product of two semiassociative algebras is again a semiassociative algebra.   The special case of tensoring an associative central division algebra and a nonassociative cyclic division algebra appeared already when constructing space time block codes
 \cite{PU11, P13, PS14, MO13, P16}.
The resulting algebra is a generalized nonassociative cyclic algebra (cf. \cite[p. 36]{J96} for the associative setup, where $D$ is assumed to be a division algebra, but the proof goes through verbatim if not):

Let $E/F$ be a cyclic field extension of degree $m$  with ${\rm Gal}(E/F)=\langle \tau\rangle$.
Let $D_0$ be a central simple algebra over $F$ of degree $n$, and put $D=D_0\otimes_FE$. 
 Let  $D_1=(E/F,\tau,d)$ be a nonassociative cyclic algebra over $F$ of degree $m$  (i.e. $d\in E^\times$), and let $\widetilde{\tau}$ be the unique extension of $\tau$ to $D$ such that $\widetilde{\tau}|_{D_0}=id_{D_0}$.  Then
$$A=D_0\otimes_F (E/F,\tau,d)\cong (D_0\otimes_F E)[t,\widetilde{\tau}]/(D_0\otimes_F E)[t,\widetilde{\tau}](t^m-d)=(D,\widetilde{\tau},d)$$
\cite{P.16} and is a nonassociative generalized cyclic algebra over $F$ of degree $mn$. Now $(E/F,\tau,d)$ is associative if and only if $d\in F^\times$, so assume that $d\in E\setminus F$. Then ${\rm Nuc}((D,\widetilde{\tau},d))=D_0\otimes_F E=D$ is a normal algebra over $F$
and $D_0\otimes_F (E/F,\tau,d)$ is a semiassociative algebra over $F$ of degree $mn$ that  is not semicentral.

For any maximal \'etale subalgebra $L$ in $D_0$, $K=L\otimes_{F} E \subset {\rm Nuc}(A)$  is a maximal \'etale subalgebra of $A$ of degree $mn$ over $F$.

\begin{proposition}
Let $H=\{\gamma\in {\rm Gal}(E/F)\,|\,\gamma(d)=d \}=\langle \tau^s\rangle$ for some integer $s$ such that $m=sr$.
\\ (i) Let
$M={\rm Fix}(\tau^s)$. Then
 $${\rm Nuc}_r((D,\widetilde{\tau},d))=D_0\otimes_F (K/M,\sigma^s, d)\cong D\oplus Dt^s\oplus \dots\oplus Dt^{s(r-1)}$$
 with $(K/M,\sigma^s, d)$ a cyclic algebra of degree $r$ over $M$,  where $[M:F]=|H|$. If $m$ is prime then ${\rm Nuc}_r((D,\widetilde{\tau},d))=D$.
 \\ (ii)  If $1<s<m$, then
$(D,\widetilde{\tau},d)^{op}$ is not a generalized nonassociative cyclic algebra.
 \end{proposition}

\begin{proof}
(i) We have  ${\rm Nuc}_r((D,\widetilde{\tau},d))=D_0\otimes_F {\rm Nuc}_r((E/F,\tau,d))$, where
${\rm Nuc}_r((E/F,\tau,d))$ is the cyclic subalgebra $(E/M,\tau^s, d)$ of degree $r$ over $M={\rm Fix}(\tau^s)$.
Here, ${\rm Nuc}_r((K/F,\tau, d))=(K/M,\sigma^s, d)$
 is a cyclic algebra of degree $r$ over $M={\rm Fix}(\tau^s)$, and $[M:F]=|H|$.
 In particular, if $m$ is prime then ${\rm Nuc}_r((K/F,\tau, d))=D_0\otimes_FE=D$.
 \\ (ii) The proof is straightforward, we just compare the left nuclei. Alternatively, this is a consequence of Theorem \ref{thm:7}, respectively, Corollary \ref{cor:8}.
\end{proof}

\begin{remark}
(i) Let $K$ be an \'etale algebra of dimension $n$ over $F$. Note that the associative algebra $D_0\otimes_F (K/F,\sigma, 0)$ also is a semiassociative algebra over $F_0$. It is not simple, the semisimple quotient is $D$.
 Abusing notation we denote it by $(D_0\otimes_F K,\sigma, 0)$ (cf. \cite[Remark 3.8]{BHMRV} for $n=2$).
 \\ (ii) The nucleus of $(D,\widetilde{\tau},d) \otimes_F(D,\widetilde{\tau},d)^{op}$ is $M_{n^2m^2}(K)$, so this algebra is not split.
 \end{remark}

 In the following, let $L/F$ be a cyclic Galois field extension of degree $n$ with ${\rm Gal}(L/F)=\langle \sigma\rangle$.
Let $K=L\otimes_{F} E$, then $\sigma$ and $\tau$ canonically extend to $K$.

Let now $D_0=(L/F,\sigma,c)$ be an associative cyclic algebra over $F$ and $D_1=(E/F,\tau,d)$
be a nonassociative cyclic algebra over $F$, i.e. $c\in F^\times$.
Then $D=(L/F,\sigma,c)\otimes_{F} E$ is a central simple algebra over $E$ of degree $n$ and $K/E$ is a maximal \'etale subalgebra of $D$ of degree $n$ (i.e., $D\cong(K/E,\sigma,c)$ using our new generalized definition of a cyclic algebra). Then
$$A=(L/F,\sigma,c)\otimes_{F} (E/F,\tau,d)$$
is a semiassociative algebra over $F$ of degree $mn$,  and if $d\in E\setminus F$ then ${\rm Nuc}(A)=D_0\otimes_F E=D$ and $K=L\otimes_{F} E \subset {\rm Nuc}(A)$  is a maximal commutative subalgebra of $A$ that is an \'etale algebra of degree $mn$ over $F$ (if $d\in F^\times$ then $(E/F,\tau,d)$ is associative).

In this case $\widetilde{\tau}$ is the unique $L$-linear automorphism of $D$ such that $\widetilde{\tau}|_K=\tau$, i.e.
 $\widetilde{\tau}(x)=\tau(x_0) +  \tau(x_1)t + \tau(x_2)t^2 +\dots + \tau(x_{n-1})t^{n-1}$
  for $x= x_0 + x_1t + x_2t^2 +\dots + x_{n-1}t^{n-1}\in D$ ($x_i\in K$, $1\leq i\leq n$).
  So here $\widetilde{\tau}|_E$ has order $m$ and ${\rm Fix}(\widetilde{\tau})=F$.

 \begin{corollary} \label{main}
The algebra
$(L/F,\sigma,c)\otimes_{F} (E/F,\tau,d)\cong (D,\widetilde{\tau},d)$
is a generalized nonassociative cyclic algebra over $F$ if $d\in F^\times$. It is associative if $d\in F$, and has nucleus $D$ if $d\in E\setminus F$.
It is $K$-semiassociative of degree $mn$ and, if it is not associative, it is not semicentral.
\end{corollary}

\begin{remark}
The proof of this result  is a straightforward generalization of the proof \cite[Theorem 11]{P16} which assumed that $D$ is a division algebra, and that $L$ and $E$ are linearly disjoint over $F$, so that  $K$ is a field: the proof that
 $(L/F,\sigma,c)\otimes_{F} (E/F,\tau,d)\cong (D,\widetilde{\tau},d) D[t;\widetilde{\tau}]/ D[t;\widetilde{\tau}](t^m-d)$
 goes through verbatim in our more general setting, as the whole theory does not depend on these two assumptions
 (it was originally developed for space-time block codes which are built from division algebras). Since we look at the opposite cyclic algebras than the one employed throughout \cite[Theorem 11]{P16}, $\widetilde{\tau}^{-1}$ in \cite[Theorem 11]{P16} in our setup becomes $\widetilde{\tau}$. If $d\in E\setminus F$ the algebra has as the nucleus the central simple algebra $D=(L/F,\sigma,c)\otimes_{F} E$, and $K/F$ is a maximal \'etale subalgebra of the nucleus of degree $mn$.
\end{remark}

Let $(E/F,\tau,d)$ be a cyclic associative division algebra of prime degree $m$. Suppose that $B_0$
 is a central associative algebra over $F$ such that
$B=B_0\otimes_{F}E$ is a division algebra. By a classical result by Jacobson, the tensor product $B_0\otimes_{F} (E/F,\tau,d)$
is a division algebra if and only if
$d\not=\widetilde{\tau}^{m-1}(z)\cdots\widetilde{\tau}(z)z$ for all $z\in B$ (\cite[Theorem 1.9.8]{J96}, see also \cite[Theorem 12, Ch. XI]{A}).

This result can be generalized to the tensor product of a cyclic and a nonassociative cyclic algebra, if the base field contains a suitable root of unity \cite{P16}.
 We now put the main results from  \cite{P16} into the context of semiassociative algebras, adjusting them where needed (some of the algebras studied in \cite[Section 3]{P16}  are the opposite algebras of ours).

The generalization of Jacobson's condition is a necessary condition for $d\in E^\times$ in our general nonassociative case as well:

\begin{proposition} \cite[Proposition 20]{P16}
Let $D_0=(L/F,\sigma,c)$ be an associative cyclic  algebra of degree $n$ over $F$,
such that $D=D_0\otimes_{F}E$ is a division algebra.
 If $D_0\otimes_{F} (E/F,\tau,d)$ is a division algebra then
 $d\not=z\widetilde{\tau}(z)\cdots \widetilde{\tau}^{m-1}(z)$
 for all $z\in D$.
\end{proposition}

From now on, let $L$ and $E$ be linearly disjoint over $F$, then $K=L\otimes_{F} E=L\cdot E$ is the composite of $L$ and $E$ over $F$ with Galois group
${\rm Gal}(K/F)=\langle \sigma\rangle\times \langle \tau\rangle$, and $K/E$ is a maximal separable subfield of $D=(L/F,\sigma,c)\otimes_{F} E$ of degree $n$.

\begin{theorem} (\cite[Theorems 14, 15, 16]{P16})\label{main2}
Suppose that $D=(L/F,\sigma,c)\otimes_{F} E$ is a division algebra, $m$ is prime and in case $m\not=2,3$, additionally that $F$ contains a primitive $m$th root of unity.
\\ (i) $(L/F,\sigma,c)\otimes_{F} (E/F,\tau,d)$ is a semiassociative division algebra  if and only if
$$d\not=z\widetilde{\tau}(z)\cdots \widetilde{\tau}^{m-1}(z)$$
 for all $z\in D$, if and only if $t^m-d\in D[t;\widetilde{\tau}]$ is irreducible.
\\ (ii) If $\tau(d^n)\not=d^n$ then $(L/F,\sigma,c)\otimes_{F} (E/F,\tau,d)$ is a  $K$-semiassociative division algebra of degree $mn$ with nucleus $D$.
\\ (iii) If $d\in E$ such that $d^n\not\in N_{D/F}(D^\times)$, then $(L/F,\sigma,c)\otimes_{F} (E/F,\tau,d)$
 is a  $K$-semiassociative division algebra  of degree $mn$ with nucleus $D$. In particular, for all $d\in E\setminus F$ with $d^n\not\in F$, $(L/F,\sigma,c)\otimes_{F} (E/F,\tau,d)$ is a $K$-semiassociative division algebra of degree $mn$.
\end{theorem}

\begin{theorem} \label{biquat} \cite[Theorem 17]{P16}
 Let $F$ be of characteristic not 2.
 Let $(a,c)_{F}$ be a quaternion algebra over $F$ which is a division algebra over $E=F(\sqrt{b})$,
 and $(F(\sqrt{b})/F,\tau,d)$ a nonassociative quaternion algebra over $F$.
 Then
$(a,c)_{F}\otimes_{F} (F(\sqrt{b})/F,\tau,d)$
 is a semiassociative division algebra over $F$ of degree 4 with nucleus $(a,c)_{F(\sqrt{b})}$.
 \end{theorem}

 More generally, let $B$ be a central simple algebra over $C$ of degree $n$, and $\sigma\in {\rm Aut}(B)$ such that
$\sigma|_{C}$ has finite order $m$, $F={\rm Fix}(\sigma)\cap C$ and $C/F$ is a cyclic Galois field extension of
 degree $m$ with $\mathrm{Gal}(C/F) = \langle \sigma|_{C} \rangle$ and $d\in C^\times$.
 Let $D_0$ be a central simple algebra over $F$ of degree $s$,
 and let $\widetilde{\sigma}$ be the unique extension of $\sigma$ to $D_0\otimes_{F_0} B$ such that $\widetilde{\sigma}|_{D_0}=id_{D_0}$
 Then $\widetilde{\sigma}$ has order $m$ over $F_0$, and
 $$D_0\otimes_{F}(B,\sigma, d)\cong (D_0\otimes B,\widetilde{\sigma}, d)$$
 with
  $(D_0\otimes_{F} B,\widetilde{\sigma}, d)=(D_0\otimes_{F}B)[t;\widetilde{\sigma}]/(D_0\otimes_{F}B)[t;\widetilde{\sigma}](t^m-d)$.
We  get a generalized nonassociative cyclic algebra of degree $mns$
with $D_0\otimes_{F} B$ contained in its nucleus.

 %
 %

\section{The semiassociative Brauer monoid} \label{subsec:structure}

\subsection{}
The classes in $Br^{sa}(F)$ that contain the homogeneous semiassociative algebras are determined by the Brauer group and are of the kind $[B]^{sa}$ with $B$ an associative central simple algebra over $F$. In particular,
if $D$ is an associative $F$-central division algebra, then $[D]^{sa}$ is the unique element of minimal degree in the class $[D]^{sa}\in Br^{sa}(F)$ which contains the homogeneous semiassociative algebras of the kind $D\otimes_F M$, where $M$ is a skew matrix algebra
\cite[Example 14.5, Corollary 18.3]{BHMRV}. Moreover, if $F$ is a field with nontrivial Brauer group, then $Br^{sa}(F)$ has elements $[A]^{sa}$ of infinite order \cite[Corollary 20.4]{BHMRV}. From the proof of \cite[Corollary 20.4]{BHMRV}, it is clear that these elements are constructed by finding semiassociative algebras $A$, such that $\sigma(A)=F\oplus B$, where $B$ is a central division algebra over $F$ of index $p$, so the similarity class $[A]^{sa}$ contains elements that are all semicentral (although the algebras are not explicitly constructed there).

We now collect some observations on elements in the semiassociative Brauer monoid $Br^{sa}(F)$.
In particular,
$Br^{sa}(F)$ can be nontrivial even if the classical Brauer group is trivial, as we can easily conclude from our previous results:

\begin{proposition}\label{prop:order}
(i) Let $F$ be a field that has a cyclic Galois field extension $K/F$ of degree $n$, ${\rm Gal}(K/F)=\langle\sigma\rangle$.  Then $[F]^{sa}\not=[(K/F,\sigma,d)]^{sa}$  for all $d\in K\setminus F$  and so $Br^{sa}(F)$ is nontrivial. Moreover, $[(K/F,\sigma,d)]^{sa}$ has infinite order in $Br^{sa}(F)$,  i.e. the powers of these element are distinct.
\\ (ii) Let $F$ be a field that has a Galois field extension $K/F$ of degree $n$, and  $(K/F, k_0,\dots,k_{m-1})$ be any  Menichetti algebra that is not associative.
Then $[F]^{sa}\not=[(K/F, k_0,\dots,k_{m-1})]^{sa}$  and so $Br^{sa}(F)$ is nontrivial, and $[(K/F, k_0,\dots,k_{m-1})]^{sa}$ has infinite order in $Br^{sa}(F)$.
\end{proposition}

\begin{proof}
 A semiassociative algebra over $F$ of degree $kn$ is split if and only if $F^{kn}$
is contained in its nucleus as a unital subalgebra.
\\ (i) Now $A=(K/F,\sigma, d)\otimes_F\dots \otimes_F (K/F,\sigma, d)$ ($k$-times) has degree
$kn$ and nucleus $K\otimes_F K\otimes_F \dots \otimes_F K$ ($k$-times). If $K/F$ is a cyclic field extension of degree $n$ with Galois group $G$ then
 $K\otimes_F K\otimes_F \dots \otimes_F K\cong \prod_{G^{k-1}} K$, where the index set $G^{k-1}$ is the $(k-1)$-fold product of $G$.
   So clearly the \'etale algebra $F^{nk-n}$ is a unital subalgebra of the nucleus of $A$, but $F^{kn}$ is not.
   \\ (ii) $(K/F, k_0,\dots,k_{m-1})\otimes_F\dots \otimes_F (K/F, k_0,\dots,k_{m-1})$ has degree
$kn$ and nucleus $K\otimes_F K\otimes_F \dots \otimes_F K$ ($k$-times), so the assertion follows as in (i).
\end{proof}

Since $(K/F,\sigma,d)$  is not semicentral  for all $d\in K\setminus F$, it does not lie in the similarity class of any $F$-central simple algebra $B$ in $Br^{sa}(F)$, and
if $n$ is prime (or if $1,d,\dots,d^{n-1}$ are linearly independent over $F$), then $(K/F,\sigma,d)$ is always a division algebra, thus is a division algebra of smallest degree in $[(K/F,\sigma,d)]^{sa}$.

\begin{lemma}\label{le:17}
Let $K/F$ be a field extension of degree $m$ and $D$ be a central simple algebra over $F$ of degree $n$. Then
$$[D]^{sa}[(K/F,\sigma,d)]^{sa}=[(D\otimes_F K,\widetilde{\sigma}, d)]^{sa}$$
for all $d\in K$.
In particular, for all $d\in K\setminus F$ we have
$$[(K/F,\sigma,d)]^{sa}=[(M_n(K),\widetilde{\sigma}, d)]^{sa}\not=[F]^{sa}$$
 and
$$[D]^{sa}[(K/F,\sigma,d)]^{sa}=[(M_n(K),\widetilde{\sigma}, d)]^{sa}=[(K/F,\sigma,d)]^{sa}$$
if $K$ is a splitting field of $D$.
\\
Moreover, for a generalized nonassociative cyclic algebra $(B,\sigma, d)$ over $F$, we have
 $$[D]^{sa}[(B,\sigma, d)]^{sa}= [(D\otimes_F B,\widetilde{\sigma}, d)]^{sa},$$
 where $\widetilde{\sigma}$ is the unique extension of $\sigma$ to $D\otimes_{F} B$ such that $\widetilde{\sigma}|_{D}=id_{D}$.
 \\
 For a Menichetti algebra $(K/F, k_0,\dots,k_{m-1})$ over $F$, we have analogously
 $$[D]^{sa}[ (K/F, k_0,\dots,k_{m-1})]^{sa}=[ (D\otimes_F K, \widetilde{\sigma}, k_0,\dots,k_{m-1})]^{sa}$$
 and for a generalized Menichetti algebra $(B,\sigma, k_0,\dots,k_{m-1})$ over $F$, we have
 $$[D]^{sa}[(B,\sigma, k_0,\dots,k_{m-1})]^{sa}=[ (D\otimes_F B, \widetilde{\sigma}, k_0,\dots,k_{m-1})]^{sa}$$
 for all $k_i\in F$, where $\widetilde{\sigma}$ is the unique extension of $\sigma$ to $D\otimes_{F} B$ such that $\widetilde{\sigma}|_{D}=id_{D}$.
\end{lemma}

\begin{proof}
 Since $D\otimes_F(K/F,\sigma,d)\cong (D\otimes_F K,\widetilde{\sigma}, d)$ is a nonassociative generalized cyclic algebra over $F$ of degree $nm$, we obtain
$$[D]^{sa}[(K/F,\sigma,d)]^{sa}=[(D\otimes_F K,\widetilde{\sigma}, d)]^{sa}=[M_2(F)]^{sa}[(K/F,\sigma,d)]^{sa}=[(K/F,\sigma,d)]^{sa}$$
 for all $d\in K\setminus F$. In particular,
 $M_n(F)\otimes_F(K/F,\sigma,d)\cong (M_n(K),\widetilde{\sigma}, d)$
 has nucleus  $M_n(K)$. The maximal \'etale $F$-algebra in its nucleus is $K^n$. This yields the assertion that
$[F]^{sa}\not=[(M_n(K),\widetilde{\sigma}, d)]^{sa}=[M_n(F)]^{sa}[(K/F,\sigma,d)]^{sa}=[(K/F,\sigma,d)]^{sa}.$ The rest is clear.
\end{proof}

\begin{theorem}\label{thm:nucl}
Let $A$ and $A'$ be two semiasssociative algebras over $F$.
\\ (i) Let  ${\rm Nuc}(A)=K$ and ${\rm Nuc}(A')=L$ be two field extensions of $F$. If
 $[A]^{sa}=[A']^{sa} \in Br^{sa}(F)$ then $K\cong L$.
\\ (ii)
 Let $A$ and $A'$ have a simple nucleus $N$, respectively $N'$, where $N$ is an $E$-central simple algebra and $N'$ is an $E'$-central simple algebra, with $E$ and $E'$ some separable field extensions of $F$. If $[A]^{sa}=[A']^{sa} \in Br^{sa}(F)$ then  $E\cong E'$ and $[N]=[N']\in Br(E)$.
\end{theorem}

\begin{proof}
Since $A\sim A'$ we have
$A\otimes_F M_n(F;c)\cong A'\otimes_F M_s(F;c') $ for suitable skewed matrix algebras $M_n(F;c)$, $M_{s}(F;c')$.
From $\sigma(A\otimes_F M_n(F;c))\cong \sigma(A'\otimes_F M_s(F;c') )$ it follows that
$\sigma(A)\otimes_F \sigma(M_n(F;c))\cong \sigma(A')\otimes_F \sigma(M_s(F;c') )$ by \cite[Proposition 13.5]{BHMRV}.
Now $\sigma(M_n(F;c))$ and $\sigma(M_s(F;c'))$ are sums of matrix algebras over $F$ whose degrees sum up to $n$, respectively to $s$:
$\sigma(M_n(F;c))\cong M_{n_1}(F)\oplus \dots\oplus M_{n_r}(F)$, respectively $\sigma(M_s(F;c'))\cong M_{s_1}(F)\oplus \dots\oplus M_{s_j}(F)$.
\\
(i) Since $K$ and $L$ are fields we have ${\rm Nuc}(A)=K=\sigma(K)$ and ${\rm Nuc}(A')=L=\sigma(L)$.
We obtain $M_{n_1}(K)\oplus \dots\oplus M_{n_r}(K)\cong M_{s_1}(L)\oplus \dots\oplus M_{s_j}(L)$.
These decompositions are unique up to permutations of summands,
so $r=j$ and $K\cong L$.
 \\ (ii)
Here, $J({\rm Nuc}(A))=J({\rm Nuc}(A'))=0$ and so $N={\rm Nuc}(A)=\sigma(A)$ and $N'={\rm Nuc}(A')=\sigma(A')$ and the above argument yields
 $M_{n_1}(N)\oplus \dots\oplus M_{n_r}(N)\cong M_{s_1}(N')\oplus \dots\oplus M_{s_j}(N')$.
 These decompositions are unique up to permutations of summands, so $r=j$ and $M_{n_1}(N)\cong M_{n_t}(N')$ for some $t$, where $N$ is an $E$-central simple algebra and $N'$ is an $E'$-central simple algebra, with $E$ and $E'$ some separable field extensions of $F$. This implies that $E \cong E'$ as both algebras must have the same center. Moreover, then $[N]=[N']\in Br(E)$.
\end{proof}

In particular, if $K/F$ and $L/F$ are two cyclic field extensions and $[(K/F,\sigma,d)]^{sa}=[(L/F,\tau,d')]^{sa}$ then $K=L$. It is an open and seemingly non-trivial  problem, if two non-isomorphic cyclic algebras $(K/F,\sigma,d)$ and $(K/F,\sigma,d')$ which are both not associative, can lie in the same similarity class in $Br^{sa}(F)$.

\begin{corollary}
(i) Let $K/F$ and $L/F$ be two cyclic field extensions and $(K/F,\sigma,d)$, $(L/F,\sigma',d')$ be two proper nonassociative cyclic algebras.
If $K$ and $L$ are not isomorphic then 
$$[(K/F,\sigma,d)]^{sa}\not=[(L/F,\sigma',d')]^{sa}$$
 in $Br^{sa}(F)$.
\\ (ii) Let  ${\rm Nuc}(A)=K$ be a field extension of degree $n$ and ${\rm Nuc}(A')=D$ an $F$-central algebra of degree $m\geq2$.
Then $[A]^{sa}\not =[A']^{sa} $ in $Br^{sa}(F)$.
\\ (iii)  Let $D_0$, $D_0'$ be two central simple algebras, and $A=D_0\otimes_F (E/F,\tau,d)\cong (D,\widetilde{\tau},d)$ and $B=D_0'\otimes_F (E'/F,\tau',d')\cong (D',\widetilde{\tau'},d')$ with $E/F$, $E'/F$ two separable field extensions and $d\in E\setminus F$, $d'\in E'\setminus F$.
 If $[A]^{sa}=[A']^{sa} \in Br^{sa}(F)$ then $E\cong E'$ and $[D_0\otimes_F E]=[D_0'\otimes_F E]\in Br (E)$.
\end{corollary}

\begin{proof}
(i)  By Theorem \ref{thm:nuc}, we know that 
 ${\rm Nuc}((K/F,\sigma, d))=K$ and ${\rm Nuc}((L/F,\sigma',d'))=L$.
The assertion is clear now since the nuclei $K$ and $L$ are not isomorphic.
\\ (ii) The analogous argument as in Theorem \ref{thm:nucl} (i) and (ii) shows that if $[A]^{sa}=[A']^{sa} \in Br^{sa}$, then
 $M_{n_1}(K)\oplus \dots\oplus M_{n_r}(K)\cong M_{s_1}(D)\oplus \dots\oplus M_{s_j}(D)$ implies $M_{n_1}(K)\cong M_{n_tb}(D_0)$ for some $b,t$ and some $F$-central division algebra $D_0$, a contradiction.
 \\ (iii) Since $d\in E\setminus F$ and $d'\in E'\setminus F$, we have ${\rm Nuc}((D,\widetilde{\tau},d))=D_0\otimes_F E=D$, and ${\rm Nuc}((D',\widetilde{\tau'},d'))=D_0'\otimes_F E'=D'$, therefore  $[(D,\widetilde{\tau},d)]^{sa}=[(D',\widetilde{\tau'},d')]^{sa}$ implies $[D_0\otimes_F E]=[D_0'\otimes_F E']\in Br (E)$ by Theorem \ref{thm:nucl} (iii).
\end{proof}

\begin{proposition}
Let $(K_i/F,\sigma_i,d_i)$ be a proper nonassociative cyclic algebra of degree $n_i$ (i.e., $K_i$ is a Galois $C_{n_i}$-algebra over $F$ with ${\rm Aut}_F(K_i)=\langle \sigma_i\rangle$ $d_i\in K\setminus F$), $i=1,\dots,r$, and let
$$A=(K_1/F,\sigma_1,d_1) \otimes_F\cdots\otimes_F (K_r/F,\sigma_r,d_r)$$
 be their tensor product (which is a semiassociative algebra of degree $n_1\cdots n_r$).
\\ (i) The nucleus of $A$ contains the \'etale algebra  $E=K_1 \otimes_F \cdots\otimes_F K_r$.
\\  (ii) If $E$ is a split \'etale algebra then $A$ is split.
\\ (iii) If $K_1,\dots,K_r$ are linearly disjoint field extensions over $F$ (e.g. all of different prime degrees) then $A$ is a semiassociative algebra of degree $n_1\cdots n_r$ with nucleus the field extension $E/F$ of degree $n_1\cdots n_r$.
In particular, $A$ is not semicentral and not homogeneous.
\\ (iv)  If $A$ is a division algebra then $K_1,\dots,K_r$ are linearly disjoint field extensions over $F$ and $E/F$ is a field extension of degree $n_1\cdots d_r$.
In particular, $A$ is not semicentral and not homogeneous.
\end{proposition}

\begin{proof}
 (i) Since the  nucleus of $(K_i/F,\sigma_i,d_i)$ contains $K_i$ (Theorem \ref{thm:nuc}), the nucleus of $A$ contains the \'etale algebra  $E=K_1 \otimes_F \cdots\otimes_F K_r$ of dimension $n_1\cdots n_r$.
 \\ (ii) The  algebra $A$ is split if and
only if the \'etale algebra  $F^{n_1\cdots n_r}$ is contained in its nucleus \cite[Proposition 7.2]{BHMRV}. Thus  if   $E\cong F^{n_1\cdots n_r}$ is a split \'etale algebra then $A$ is split.
 \\ (iii) Since  ${\rm Nuc}((K_i/F,\sigma_i,d_i))=K_i$ (Theorem \ref{thm:nuc}),  ${\rm Nuc}(A)=K_1 \otimes_F \cdots\otimes_F K_r=E$. Since the fields $K_i$ are assumed to be linearly disjoined, $E/F$ is a field extension of degree $n_1\cdots n_r$. The only atom of $A$ is $E$ and so $A$ is not semicentral and not homogeneous.
 \\ (iv) is trivial.
\end{proof}

Mirrowing the classical setup, the semiassociative Brauer monoid of an algebraically closed field is trivial, e.g. $Br_{sa}(\mathbb{C}) = 1$  \cite[Example 14.5]{BHMRV}, so any semiassociative algebra over $\mathbb{C}$ splits.

\subsection{$Br^{sa}(\mathbb{R})$ }

We known that $Br(\mathbb{R})=\{[\mathbb{R}], [\mathbb{H}]\}$ is a cyclic group of order 2; and $\mathbb{H}\otimes_\mathbb{R} \mathbb{H}\cong M_4(\mathbb{R})$. Therefore the two classes in $Br^{sa}(\mathbb{R})$ that contain the homogeneous semiassociative algebras are $[\mathbb{R}]^{sa}$ and $[\mathbb{H}]^{sa}$.

Up to isomorphism, every nonassociative simple algebra of dimension four with $\mathbb{C}$  as its nucleus is a nonassociative quaternion algebra \cite{W} (note that $(\mathbb{C}/\mathbb{R},\can,0)$ is semiassociative, even associative, but not simple). For every $a\in \mathbb{C}\setminus \mathbb{R}$, the nonassociative quaternion algebra $(\mathbb{C}/\mathbb{R},\can,a)$ is a semiassociative division algebra over $\mathbb{R}$ of degree two that is not semicentral and not homogeneous, and
$[(\mathbb{C}/\mathbb{R},\can,a)]^{sa}$ has infinite order in $Br^{sa}(\mathbb{R})$.
The class $[(\mathbb{C}/\mathbb{R},\can,a)]^{sa}$  thus
contains algebras that are not semicentral and $(\mathbb{C}/\mathbb{R},\can,a)$  is a division algebra of smallest degree in $[(\mathbb{C}/\mathbb{R},\can,a)]^{sa}$. For $a,b\in \mathbb{C}\setminus \mathbb{R}$, we have $(\mathbb{C}/\mathbb{R},\can,a)\cong (\mathbb{C}/\mathbb{R},\can,b)$ if and only if there is $x\in \mathbb{R}$ such that either $a=x^2b$ or $\bar a=x^2b$
\cite{W}.
It is not clear, however, if two nonisomorphic quaternion division algebras can lie in the same similarity class in $Br^{sa}(\mathbb{R})$.

Furthermore, for all $d\in\mathbb{C}\setminus \mathbb{R}$ we have for $\sigma=\can$ the complex conjugation by Lemma \ref{le:17}:
$$\mathbb{H}\otimes_\mathbb{R} (\mathbb{C}/\mathbb{R},\sigma,d)\cong (\mathbb{H}\otimes_\mathbb{R} \mathbb{C},\widetilde{\sigma}, d)
\cong(M_2(\mathbb{C}),\widetilde{\sigma}, d)$$
is a proper generalized nonassociative cyclic algebra,
which implies
$$[\mathbb{H}]^{sa}[(\mathbb{C}/\mathbb{R},\sigma,d)]^{sa}=[(M_2(\mathbb{C}),\widetilde{\sigma}, d)]^{sa}\not=[\mathbb{R}]^{sa},$$
Analogously,
$$M_n(\mathbb{R}) \otimes_\mathbb{R} (\mathbb{C}/\mathbb{R},\sigma,d) \cong (M_n(\mathbb{R}) \otimes_\mathbb{R} \mathbb{C},\widetilde{\sigma}, d) \cong (M_n(\mathbb{C}),\widetilde{\sigma}, d)$$
is a proper generalized nonassociative cyclic algebra, which in turn implies
$$[M_n(\mathbb{R})]^{sa}[(\mathbb{C}/\mathbb{R},\sigma,d)]^{sa}=[M_n(\mathbb{C}),\widetilde{\sigma}, d)]^{sa}\not=[\mathbb{R}]^{sa}.$$
We conclude that, perhaps counter-intuitively, 
$$[\mathbb{H}]^{sa}[(\mathbb{C}/\mathbb{R},\sigma,d)]^{sa}=[(M_2(\mathbb{C}),\widetilde{\sigma}, d)]^{sa}=[M_2(\mathbb{R})]^{sa}[(\mathbb{C}/\mathbb{R},\sigma,d)]^{sa}=[(\mathbb{C}/\mathbb{R},\sigma,d)]^{sa}$$
in $Br^{sa}(\mathbb{R})$.
These examples also demonstrate that, unlike in the classical Brauer group,  algebras with zero divisors can represent nontrivial elements in the semiassociative Brauer monoid $Br^{sa}(\mathbb{R})$.

\subsection{$Br^{sa}(\mathbb{F}_q)$ }

 The  Brauer group $Br(\mathbb{F}_q)$ is trivial. Therefore the only class in $Br^{sa}(\mathbb{F}_q)$ that contains homogeneous semiassociative algebras is the trivial class $[\mathbb{F}_q]^{sa}=[M_n(\mathbb{F}_q;c)]^{sa}$ with $M_n(\mathbb{F}_q;c)$ a skew matrix algebra over $F$. 
  
The semiassociative Brauer monoid $Br^{sa}(\mathbb{F}_q)$ is not trivial: For each $n\geq2$, there exists a field extension $K/\mathbb{F}_q$ of degree $n$ that is unique up to isomorphism. Let $\sigma$ be a generator of its Galois group.
  Then for each such $K/\mathbb{F}_q$ of degree $n$, and for all $a\in K\setminus \mathbb{F}_q$, $(K/\mathbb{F}_q,\sigma,a)$ is a simple nonassociative cyclic algebra of degree $n$ with nucleus $K$, and is not semisimple. The class  $[(K/\mathbb{F}_q,\sigma,a)]^{sa}$ is not trivial and has infinite order in $Br^{sa}(F)$ (Proposition \ref{prop:order}). Moreover, if $\sigma'$ is another generator of the Galois group of $K/\mathbb{F}_q$, $\sigma\not=\sigma'$, then the algebras $(K/\mathbb{F}_q,\sigma,a)$ and  $(K/\mathbb{F}_q,\sigma',a')$ will not be  isomorphic for any choice of $a,a'\in K\setminus \mathbb{F}_q$ \cite[Theorem 5.1]{NevPum2024}.
   Two  algebras  $(K/\mathbb{F}_q,\sigma,a)$ and  $(K/\mathbb{F}_q,\sigma,a')$ will be isomorphic if and only if there exists some $\tau\in {\rm Gal}(K/\mathbb{F}_q)$  such that
$a \in \tau(a')N_{K/\mathbb{F}_q}(K^\times)$ \cite[Corollary 32]{BrownPumpluen2018}.

Unfortunately, we do not know if two nonisomorphic cyclic algebras can lie in the same similarity class in $Br^{sa}(\mathbb{R})$.

If $a$ does not lie in a proper subfield of $K/F$, then $(K/F,\sigma,a)$ is a division algebra \cite[Theorem 4.4, Corollary 4.5]{S12}.
Again,
  $[M_n(\mathbb{F}_q)]^{sa}[(K/\mathbb{F}_q,\sigma,d)]^{sa}=[(M_n(K),\widetilde{\sigma},d)]^{sa},$
i.e. 
$$[(K/\mathbb{F}_q,\sigma,d)]^{sa}=[(M_n(K),\widetilde{\sigma},d)]^{sa}.$$
Independently, for each such $K/\mathbb{F}_q$ of degree $n$, we also have the semiassociative algebra $(K/\mathbb{F}_q,\sigma,0)$ which is not simple, and contains zero divisors.

There exist other large classes of semifields, e.g. Menichetti algebras, to name just one, that are all semiassociative.

%
%

\section{Algebras that are not semiassociative}\label{sec:2}

 Semiassociative algebras over $F$ can be defined in
terms of simple subalgebras of the nucleus whose center is separable
over $F$ \cite[Section 5]{BHMRV}. This avoids problems when tensoring these algebras, as the tensor product of two field extensions that are not both separable may be a very ``unpleasant'' algebra, but excludes algebras with nuclei that are simple subalgebras, and whose centers are (purely) inseparable over $F$, so creates restrictions when  $char(F)=p$ is prime.
An example of such algebras that are excluded in the current set-up are nonassociative algebras of square dimension that are a canonical generalization of cyclic $p$-algebras, and of Amitsur's differential algebras  (\cite{Am, Am2, Hoe},
\cite[Sections 1.5, 1.8, 1.9]{J96}). Their nucleus is a purely inseparable field extension of $F$:

\subsection{Nonassociative differential extensions of a field}\label{sec:3}

 Let $K$ be a field of characteristic $p$ together with an algebraic derivation $\delta:K\to K$ of $K$ of degree $p$ with minimum polynomial
$g(t)=t^{p^e}+a_1t^{p^{e-1}}+\dots+ a_et\in F[t]$, where $F={\rm Const}(\delta)=\{a\in K\,|\, \delta(a)=0\}$. Put $R=K[t;\delta]$.
 Then $K/F$ is a purely inseparable extension of exponent one, and $[K:F]=p$.
Let $f(t)=g(t)-d\in K[t;\delta]$, then the nonassociative $F$-central algebra
$$(K,\delta,d)=K[t;\delta]/K[t;\delta] f$$
has dimension $p^{2e}$ and
is called a \emph{(nonassociative) differential extension} of $K$, or  a \emph{(nonassociative) differential algebra}. $(K,\delta,d)$ is a division algebra if and only if $f\in K[t;\delta]$ is irreducible.
$(K,\delta,d)$ is associative if and only if $d\in F$.

If $f(t)\in F[t]$ then $(K,\delta,d)$ is an associative central simple algebra over $F$, and $K$ is a maximal subfield of $(K,\delta,d)$ of dimension $p$ \cite[p.~23]{J96}. If $f(t)\in F[t]$ is irreducible and $g(t)=t^p-t$, then $(K,\delta,d)$ contains the cyclic separable field extension $F[t]/(t^p-t-d)$ of degree $p$, and is a cyclic $p$-algebra.

For $d\in K\setminus F$,  $(K,\delta,d)$ is not semiassociative as its nucleus is the purely inseparable field extension $K/F$ of degree $p$ \cite{P16.0}.
The
$K$-algebra $(K,\delta,d)\otimes_F K$ contains the simple truncated polynomial algebra $K\otimes_F K$ in its nucleus, and is called a \emph{split differential extension}.

%
%

\subsection{Nonassociative differential extensions of a division algebra} \label{sec:4}

 There are classes of algebras over $F$ that have a central simple algebra $D$ over a field $C$ as their nucleus, but the field extension $C/F$ is purely inseparable of degree $p$, so the center $C$ of $D$ is purely inseparable over $F$ \cite{P16.0}:

Let $C$ be a field  of characteristic $p$ and $D$ be a central simple
algebra over $C$
 of degree $n$ ($D=C$ is allowed and brings us back to the setup of the previous subsection).
Let  $\delta$ be a derivation of $D$, such that $\delta|_C$ is algebraic with minimum polynomial
$g(t)=t^{p^e}+a_1t^{p^{e-1}}+\dots+ a_et\in F[t]$ and let $F={\rm Const}(\delta)$.

 Assume that $g(\delta)=id_{d_0}$ is an inner
derivation of $D$  and that there exists $d_0\in F$  so that $\delta(d_0)=0$ (this is always possible if $D$ is a division algebra \cite[Lemma 1.5.3]{J96}). The center of $R=D[t;\delta]$ is $F[z]$ with
$z=g(t)-d_0$.

For all $f(t)=g(t)-d\in D[t;\delta] $, the nonassociative  unital $F$-algebra
$$(D,\delta,d)=S_f=D[t;\delta]/D[t;\delta] f(t)$$
 has dimension $p^{2e}n^2$ over $F$ and is called a \emph{nonassociative generalized differential algebra/extension}.
 (Note that Amitsur's  associative differential extensions of division rings $D$ were generalized to simple rings $D$ already  in \cite{K}.)
  $(D,\delta,d)$ is an associative algebra if and only if $d\in F$ \cite[Theorem 20]{P16.0}.
For $d\in F$, $(D,\delta,d)$ is a central simple associative algebra over $F$ (cf. \cite[p.~23]{J96} if $D$ is a division algebra.)
  If  $f(t)=g(t)-d\in F[t]$ is irreducible, then $(D,\delta,d)$ contains the field extension $F[t]/(g(t)-d)$    as a subfield.

If $d\in C\setminus F$, then the differential algebra $(C,\delta|_C,d)$ is a subalgebra of $(D,\delta,d)$ of dimension $p^{2e}$.

We now correct part of \cite[Lemma 19]{P16.0}:

\begin{lemma} \label{le:2}
For $d\in C\setminus F$, $D\subset {\rm Nuc}_r(D,\delta,d)$, thus $D\subset{\rm Nuc}((D,\delta,d))$. If $D$ is a division ring, the inclusions become equalities.
\end{lemma}

\begin{proof}
 Since $g$ is semi-invariant and monic of minimal degree, we have $g(t)a=ag(t)$ for all $a\in D$ (\cite[(2.1), p.~3]{LLLM}, this also holds when $D$ is only simple and not a division algebra, since the polynomial is monic). Since $d\in C\setminus F$, we get $f(t)a=ag(t)-ad=af(t)$ for all $a\in D$ and so in turn $f$ is semi-invariant as well.
Hence the right nucleus of $(K,\delta,d)$ contains $D$  \cite[Proposition 3]{P16.0}.
\end{proof}

Thus every maximal \'etale subalgebra $N$ of $D$  also lies in the nucleus and has dimension $p^en$ as algebra over $F$.

As an algebra over $F$, $N$ is the product of finite dimensional field extensions that are each of the type $N_i/F$, where
we have a tower of field extensions $F\subset C\subset N_i$, such that $N_i/C$ is separable of degree $n$ and $C/F$ purely inseparable of exponent one. This means we can write every $N_i$ as a tensor product $N_i=S_i\otimes_F C $, where $S_i$ is the maximal separable subfield of $N_i/F$ \cite[p. 32]{J96}, and obtain that
$$N=N_1\times \dots \times N_r=(S_1\otimes_F C )\times \dots \times (S_r\otimes_F C )=(S_1\times \dots \times S_r)\otimes_F C$$
 is an \'etale algebra $S_1\times \dots \times S_r$ over $F$ tensored over $F$ with the purely inseparable field extension $C/F$ of exponent one.

\begin{remark}
Let $R=D[t;\delta]$ and define
$V_p(b)=b^p+\delta^{p-1}(b)+*$
 for all $b\in D$, with $*$ a sum of commutators of $b$, $\delta(b),\dots, \delta^{p-2}(b)$, that is
$V_2(b)=b^2+\delta(b),$ $ V_3(b)=b^3+\delta^{2}(b)+[\delta(b),b],  $
 and so on  \cite[p.~18, (1.3.20)]{J96}.  If $D$ is commutative, or if $b\in D$ commutes with all its derivatives, then the sum $*$  is 0 and the the expression simplifies to
$V_p(b)=b^p+\delta^{p-1}(b)$
 \cite[p.~17ff]{J96}. It is easy to see that
$(t-b)^p=t^p-V_p(b)$
 for all $b\in D$ \cite[1.3.19]{J96}.
 
It is well-known that  $V:C\rightarrow F$, $V(a)=V_{p}(a)-a$ is a homomorphism between the additive groups
$C$ and $F$ \cite{J37}.
When $D$ is a division algebra then $(D,\delta,d)$  is a division algebra if and only if $f$ is irreducible, if and only if $d\not=V_p(z)-z$ for all $z\in D$, if and only if $d\not=(t-z)^p-t^p-z$ for all $z\in D$.

Let $F$ have characteristic 3, and $\delta$ have minimum polynomial
$g(t)=t^3-ct\in F[t].$ Then for $f(t)=t^3-ct-d\in C[t;\delta]$, $(D,\delta,d)$
is a  unital algebra over $F$ of dimension $9$, and a division algebra  if and only if
$V_3(z)-cz\not=d$  and $ V_3(z)-zc-d+\delta(c)\not=0$
 for all $z\in D$
 \cite[Theorem 23]{P16.0}.
\end{remark}

Let $F$ be a field of characteristic $p$ and $D$ be a central simple algebra over $F$ of degree $n$.
Let $K/F$ be a purely inseparable extension of exponent one such that $[K:F]=p$. Let $\delta$ be a derivation on $K$ with $F={\rm Const}(\delta)$, such that
$\delta$ is an algebraic derivation of degree $p$ with minimum polynomial $g(t)=t^p-t\in F[t]$
of degree $p$.
Let $\delta$ be the extension of $\delta$ to $D_K$ such that $\delta|_D=0$.
 Then $(K,\delta,d)\otimes_F D\cong (D_K,\delta,d)$ is an algebra of dimension $n^2p^2$ over $F$.

 Moreover, if $D_K=D\otimes_F K$ is a division algebra and $(K,\delta,d)$ is a division algebra over $F$, then
$(K,\delta,d)\otimes_F D\cong (D_K,\delta,d)$ is a division algebra if and only if $f(t)=g(t)-d $ is irreducible in $D_K[t;\delta]$.
  In particular for $g(t)=t^p-t$ this is the case, if and only if $d\not=V_p(z)-z$ for all $z\in D_K$, if and only if $d\not=(t-z)^p-t^p-z$ for all $z\in D_K$  \cite{P16.0}.

%
%

\section{Outlook}

While there are excellent reasons to use the existing definition of $Br^{sa}(F)$ (it is the broadest possible one if we want to use Brauer factor sets), we believe it makes sense to discuss (i) a possible refinement of the semiassociative Brauer monoid to include only simple semiassociative algebras, and (ii) a possible generalization of $Br^{sa}(F)$ that allows us to include nonassociative differential algebras, if the base field $F$ is not perfect and has characteristic $p$:
\\ (i)
The simple semiassociative algebras form a submonoid of $Br^{sa}(F)$ that still contains $Br(F)$ as unique maximal subgroup. If we only consider the simple semiassociative algebras we exclude pathological cases like the associative algebras $(K/F,\sigma, 0)$ or $(D,\sigma, 0)$.
\\ (ii) Suppose that we want to include generalized differential extensions in the definition of the Brauer monoid. Here is one possible way to proceed:
Let $A$ be an $F$-central nonassociative algebra over $F$ of dimension $l^2$, ${\rm char}(F)=p$. We call $A$ a \emph{generalized semisassociative algebra}, if its nucleus contains a tensor product $N=N_1\otimes_F\dots\otimes_F N_s$ of finite field extensions $N_i/F$ such that $dim_F N=l$, with $N_i$  either separable or purely inseparable of exponent one and $N_i/F$ primitive of the kind $N_i=F[x]$ for $x_i^{p}=a\in F$, and if  $A$ is cyclic and faithful as an $N\otimes N$-module.
The root $l$ of the dimension of $A$ is called the \emph{degree} of $A$.

If all $N_i/F$ are separable then $N$ is an \'etale algebra over $F$, so in this case $A$ is semiassociative in the sense of \cite{BHMRV}.

Two generalized semiassociative algebras $A$ and $B$ over $F$ are called \emph{Brauer equivalent}, if there exist skew matrix algebras $M_n(F;c)$ and $M_m(F;c')$ such that $A\otimes_F M_n(F;c) \cong B \otimes_F M_m(F;c') $. This is an equivalence relation, as \cite[Remark 6.9]{BHMRV} still holds. We denote the equivalence class of a generalized semiassociative algebra $A$ by $[A]^{gsa}$ and the monoid of equivalence classes by $Br^{gsa}(F)$.

 Note that every finite purely inseparable field extension of exponent one is a tensor product of primitive extensions $F[x_1]\otimes \cdots\otimes F[x_r]$, where $x_i^{p}-a_i=0$. Also note that if $N$ is purely inseparable of exponent one, then $N\otimes N$ is a truncated polynomial algebra that is isomorphic to $F[G]$ for a finite abelian $p$-group $G$ \cite{C}.
 This means that $N$ is always the tensor product of an  \'etale algebra over $F$ (the ``separable part'') and another algebra (the ``inseparable part''). This inseparable part is either a purely inseparable field extension of $F$ of exponent one, an $F$-algebras $F[G]$,  or a tensor product of an $F$-algebra $F[G]$ and a purely inseparable field extension of $F$ of exponent one.

 Let $A$ be a generalized semiassociative algebra of degree $n$ with $N\subset {\rm Nuc}(A)$, $N=N_1\otimes_F\dots\otimes_F N_s$ of dimension $n$ over $F$, with finite field extensions $N_i/F$ where $N_i$  either separable or purely inseparable and primitive of exponent one. Then $A$ is called \emph{split}, if $N\cong E\otimes_F F[G]$, where $E/F$ is a split \'etale algebra,  and $F[G]$ is a simple truncated polynomial algebra ($G$ an abelian $p$-group). We allow here that $E=F$ or $F[G]=F$. A finite-dimensional field extension $E/F$ \emph{splits} $A$, if $N\otimes_F E\cong S\otimes_E E[G]$ for a suitable abelian $p$ group $G$, and an \'etale algebra $S$ over $E$. We also note that if $N=N_1\otimes_F N_2$ is the tensor product of a separable and a purely inseparable extension of exponent one, then $N/F$ is a finite field extension, as $N_1$ and $N_2$ are linearly disjoint over $F$.

 Let $A$ be a generalized semiassociative algebra with $K ={\rm Nuc}(A)$ a purely inseparable field extension of exponent one. Then a finite-dimensional field extension $E/F$ splits $A$, if $K\otimes_F E$ is a simple truncated polynomial algebra, which is the case if and only if $F\subset K\subset E$ is an intermediate field. In particular, $K$ splits $(K,\delta,d)$. For the nonassociative generalized differential algebra $(D,\delta,d)$ we know that if $K$ is a purely inseparable splitting field of the $C$-central simple algebra $D$, then $(D,\delta,d)\otimes_F K\cong (M_n(K),\delta,d)$ is a generalized semiassociative algebra over $K$ whose
  nucleus contains the $K$-algebra $K\otimes_F K\cong F[G]$, but the algebra is not split, as the dimension of $F[G]$ is too small.

 Moreover, for every central simple algebra $D$ over $F$, we have $D\otimes_F (K,\delta,d)  \cong (D_K,\delta,d) $,
 so if $K$ is a purely inseparable splitting field of $D$, then $D\otimes_F (K,\delta,d)  \cong (M_n(K),\delta,d)$, which is, however, not a split algebra over $F$.

 Let $D$ be a $p$-algebra of degree $p^s$ with maximal separable splitting field $E$ and purely inseparable simple splitting field $L$ of degree $p^f\leq p^e$ or degree $p^s$ (so $D$ is cyclic). Then  $D\otimes_F (K,\delta,d)\cong (D\otimes_F K, \delta,d)$ contains the field extension $E\otimes_F K\subset D\otimes_F K$ of degree $p^{s}p=p^{e+1}$ and the algebra $L\otimes_F K\subset D\otimes_F K$ in its nucleus. Here, $L\otimes_F K$ is either a finite purely inseparable field extension of exponent one, or - if $L=K$ - an algebra $F(G)$. In the later case, we get $D\otimes_F (K,\delta,d)\cong (M_{p^s}(K), \delta,d)$.

Alternatively, we could define  $Br^{gsa}(F)$ as the submonoid of the above generalized one that is generated by $Br^{sa}(F)$ and the algebras
$(D,\delta,d)$ and $(K,\delta,d)$, and call the resulting algebras \emph{generalized semiassociative algebras}.

In either case, we obtain that
$$[(K,\delta,d)]^{gsa}[D]^{gsa}=[(D_K,\delta,d)]^{gsa},$$ in particular
$[(K,\delta,d)]^{gsa}= [(K,\delta,d)]^{gsa}[M_n(F)]^{gsa}= [(M_n(K),\delta,d)]^{gsa}$.
Furthermore, if $D$ is a $p$-algebra of degree $p^s$ and $K$ is a finite-dimensional purely inseparable splitting field of $D$ then
$$[(K,\delta,d)]^{gsa}[D]^{gsa}=[(M_{p^s}(K),\delta,d)]^{gsa}.$$

It would be interesting to explore other relations.
\\\\
{\emph Acknowledgement:} Parts of this paper were completed while the author was a visitor at the University of Ottawa. She acknowledges support from Monica Nevins' NSERC Discovery Grant RGPIN-2020-05020. She would like to thank the Department of Mathematics and Statistics for its hospitality. She would also like the anonymous referee for the careful review and their suggestions that helped improve the paper.


\end{document}